\documentclass[11pt]{amsart}

\usepackage[T1]{fontenc}
\usepackage{amsmath,amsfonts}
\usepackage{amssymb}
\usepackage{amscd}
\usepackage{amsthm}
\usepackage{yhmath}
\usepackage{subfigure}
\usepackage{graphicx}   

\usepackage[all]{xy}
\usepackage{color}

\usepackage{marginnote}
\usepackage{hyperref}

\usepackage{todonotes}
\numberwithin{equation}{section}

\newtheorem{theorem}[equation]{Theorem}
\newtheorem{thm}[equation]{Theorem}
\newtheorem{proposition}[equation]{Proposition}

\newtheorem{lemma}[equation]{Lemma}

\newtheorem{corollary}[equation]{Corollary}

\theoremstyle{remark}

\theoremstyle{definition}

\newtheorem{Rem}[equation]{Remark}  %REM
\newtheorem{remark}[equation]{Remark}
\newtheorem{Def}[equation]{Definition} % DEF
\newtheorem{exercise}[equation]{Exercise}

\newtheorem{definition}[equation]{Definition}

\newtheorem{example}[equation]{Example}

\def\XXint#1#2#3{{\setbox0=\hbox{$#1{#2#3}{\int}$}
	\vcenter{\hbox{$#2#3$}}\kern-.5\wd0}}

\newcommand{\ra}{\rightarrow}

\newcommand{\N}{\mathbb N}

\newcommand{\C}{\mathbb C}

\newcommand{\R}{\mathbb R}

\newcommand{\g}{\mathfrak{g}}

\newcommand{\h}{\mathfrak{h}}

\newcommand{\Span}{\operatorname{span}}
\newcommand{\acts}{\curvearrowright}

\newcommand{\norm}[1]{\left\Vert#1\right\Vert}

\def\eps{\epsilon}

\def\ra{\rightarrow}

\newcommand{\vol}{\operatorname{vol}}

\begin{document}

\title[A primer on Carnot groups]{A primer on Carnot groups:\\
{\small 
homogenous groups,
CC spaces, and regularity of their isometries}}
%structure of graded groups,
% homogenous distances,  and isometries}}

\author{Enrico Le Donne}

\address[Le Donne]{Department of Mathematics and Statistics, P.O. Box 35,
FI-40014,
University of Jyv\"askyl\"a, Finland}
\email{ledonne@msri.org}

\begin{abstract}
% Mini course hold at the
% Ninth School on `Analysis and Geometry in Metric Spaces'
%in Levico Terme (Trento), 6--10 July 2015.
Carnot groups are distinguished spaces that are rich of structure: they are those Lie groups equipped with a path distance that is invariant by left-translations of the group and admit automorphisms that are dilations with respect to the distance.
We present the basic theory of Carnot groups together with several remarks.
We consider them as special cases of graded groups and as homogeneous metric spaces.
We discuss the regularity of isometries in the general case of Carnot-Carath\'eodory spaces and of nilpotent metric Lie groups. 
\end{abstract}

\renewcommand{\subjclassname} {\textup{2010} Mathematics Subject Classification}
\subjclass[]{
53C17, %   Sub-Riemannian geometry
%28C15, %  	Set functions and measures on topological spaces (regularity of measures, etc.)
%49Q15, %  Geometric measure and integration theory, integral and normal currents
43A80, % (1973-now) Analysis on other specific Lie groups
%53C60,   % Finsler spaces and generalizations 
%28A75,  %  Length, area, volume, other geometric measure theory
%26A16  % Lipschitz (Hlder) classes
%58C35   Integration on manifolds; measures on manifolds
%26B20 Integral formulas (Stokes, Gauss, Green, etc.)
%54Exx, % Spaces with richer structures 
%37L40 %Invariant measures
%58D05, %Groups of diffeomorphisms and homeomorphisms as manifolds
%22F50, %Groups as automorphisms of other structures
% 22DXX % Locally compact groups and their algebras
22E25, % Nilpotent and solvable Lie groups
 22F30, % Homogeneous spaces
14M17. %Homogeneous spaces and generalizations 
% 53C30 % Homogeneous manifolds
% 58D19% Group actions and symmetry properties
% 58C25 % Differentiable maps
%49J20  % (1991-now) Optimal control problems involving partial differential equations
%49K21,  % (2010-now) Problems involving relations other than differential equations
%49J15  %(1991-now) Optimal control problems involving ordinary differential equations
}
 
\thanks{This essay is a survey written for a summer school on metric spaces.}
\date{\today}
\maketitle
\tableofcontents

%Carnot groups are also known as stratified groups.
Carnot groups are 
%particular
special
%distinguished
%peculiar 
cases of Carnot-Carath\'eodory spaces associated with a system of bracket-generating vector fields. In particular, they are geodesic metric spaces. 
Carnot groups are also examples of homogeneous groups and stratified groups.
They are simply connected nilpotent groups and their Lie algebras admit  special gradings:  stratifications.
The stratification can be used to define a left-invariant metric on each group in such a way that with respect to this metric the group is self-similar.
Namely, there is a natural family of dilations on the group under which the metric behaves like the Euclidean metric under Euclidean dilations.

Carnot groups, and more generally homogeneous groups and Carnot-Carath\'eodory spaces, appear in several mathematical contexts.
Such groups appear in harmonic analysis, in the study of hypoelliptic differential operators, as boundary of strictly pseudo-convex complex domains, see the books \cite{Stein:book, Capogna-et-al} as initial references.
Carnot groups, with subFinsler distances, appear in geometric group theory as asymptotic cones of nilpotent finitely generated groups, see
\cite{Gromov1, Pansu}.
SubRiemannian Carnot groups are limits of Riemannian manifolds and are metric tangents of  subRiemannian manifolds. 
SubRiemannian geometries arise  in many areas of pure  and
applied  mathematics (such as algebra, geometry, analysis, mechanics, control theory, mathematical physics), as well as
 in applications (e.g., robotics), for references see the book \cite{Montgomery}. 
The literature on geometry and analysis on Carnot groups is plentiful. In addition to the previous references, we also cite some among the more authoritative ones \cite{
Roth:Stein,
Folland-Stein,
nagelstwe,
Koranyi-Reimann85,
%Pansu,
varsalcou,
%Stein:book,
Heinonen-calculus,
%Gromov1,
magntesi,
Vittone-tesi,
Bonfiglioli:et:al,
%Capogna-et-al,
jeancontrol,Rifford:book,
Agrachev_Barilari_Boscain:book}.

The setting of Carnot groups has both similarities and differences compared to the Euclidean case. 
In addition to the geodesic distance and the presence of dilations and translations,  on each Carnot group one naturally considers Haar measures, which are unique up to scalar multiplication and are translation invariant. Moreover, in this setting they have the important property of being Ahlfors-regular measures. In fact, the measure of each $r$-ball is the measure of the unit ball multiplied by $r^Q$, where $Q$ in a fixed integer. With so much structure of metric measure space, it has become highly interesting to study geometric measure theory, and other aspects of analysis or geometry, in Carnot groups.
On the one hand, with so much structure many results in the Euclidean setting generalizes to Carnot groups. The most celebrated example is
Pansu's version of Rademacher's theorem for Lipschitz maps (see \cite{Pansu}). 
Other results that generalize are 
the Isoperimetric Inequality \cite{pansucras},
the Poincar\'e Inequality \cite{Jerison},
the Nash Embedding Theorem \cite{LeDonne_PIE}, 
the Myers-Steenrod Regularity Theorem \cite{Capogna_LeDonne},
 and, at least partially, De-Giorgi Structure Theorem \cite{fssc}.
On the other hand, Carnot groups exhibit fractal behavior, the reason being that on such metric spaces the only curves of finite length are the horizontal ones. In fact, 
except for the Abelian groups, 
even if the distance is defined by a smooth subbundle,   it is not smooth and the Hausdorff dimension
 differs from the topological dimension. Moreover, these spaces contain no subset of positive measure that is bi-Lipschitz equivalent to a subset of a Euclidean space   and do not admit biLipschitz embedding into reflexive Banach spaces, nor into $L^1$, see \cite{Semmes2,Ambrosio-Kirchheim,Cheeger-Kleiner,Cheeger-Kleiner2,
Cheeger-Kleiner3}.
For this reason, Carnot groups, together with the classical fractals and boundaries of hyperbolic groups, are the main examples in an emerging field called `non-smooth analysis', in addition see 
 \cite{Semmes, Heinonen-Koskela, Cheeger,  Lang-Plaut, Heinonenbook, Laakso,Gromov-Schoen, Burago-Gromov-Perelman,
Cheeger-Colding, Bestvina-Mess, Margulis-Mostow, Kleiner-Leeb, Gromov, Bourdon_Pajot_2000, Kapovich-Benakli,
Bonk-Kleiner}.
A non-exhaustive, but long-enough list of other contribution in geometric measure theory on Carnot groups  is \cite{PansuThesis,
Koranyi-Reimann85,
Koranyi-Reimann95,
Ambrosio-Kirchheim:rect, HK,
Amb02,
Heinonenbook,
 Amb01, 
Pauls-minimal, 
Kirchheim-SerraCassano,
colepauls,
  ASCV, 
   Barone-Adesi-Serra-Cassano-Vittone,
 MSCV, 
 Danielli-Garofalo-Nhieu,
 Magnani-Vittone,
 Magnani08,
 Ritore-Rosales08,
Ritore, 
 AKL,
MR2531368,
LeDonne-Zust,
Bellettini-LeDonne,
LeDonne_Rigot_Heisenberg_BCP,
LeDonne_NicolussiGolo,
LeDonne_Rigot_BCP_graded_groups}, and more can be found in \cite{SerraCassanoGMT, LeDonne-lectures}.

In this essay, we shall distinguish between Carnot groups, homogeneous groups, and stratified groups. In fact, the latter are Lie groups with a particular kind of grading, i.e.,  a stratification. Hence, they are only an algebraic object. Instead, homogeneous groups are Lie groups that are arbitrarily graded but are equipped with distances that are one-homogeneous with respect to the dilations induced by the grading. Finally, for the purpose of this paper, the term Carnot group will be reserved for those  stratified groups that are equipped with homogeneous distances making them Carnot-Carath\'eodory spaces. It will be obvious that up to biLipschitz equivalence every Carnot group has one unique geometric structure. This is the reason why this term is sometimes replacing the term stratified group.

From the metric viewpoint,
Carnot groups are peculiar examples of isometrically homogeneous spaces.
In this context we shall  differently use the term `homogenous': it means the presence of a transitive group action.
In fact, on Carnot groups  left-translations act transitively and by isometries.
In some sense, these groups are
quite prototypical examples of geodesic homogeneous spaces.
An interesting result of Berestovskii gives that every 
isometrically homogeneous space whose distance is geodesic is indeed a Carnot-Carath\'eodory space, see  Theorem~\ref{thm:Berestovskii}.
Hence, Carnot-Carath\'eodory spaces and Carnot groups are natural examples in metric geometry. This is the reason why they appear in different contexts.

A purpose of this essay is to present the notion of a Carnot group from different viewpoints.
As we said, Carnot groups have very rich metric and geometric structures.
However, they are easily characterizable as geodesic spaces with self-similarity. Here is an  equivalent definition, which is intermediate between the standard one (Section~\ref{sec:Carnot}) and one of the simplest axiomatic ones (Theorem~\ref{characterization_Carnot}).
A {\em Carnot group} is a Lie group $G$ endowed with a left-invariant geodesic distance $d$ admitting for each $\lambda>0$ a bijection $\delta_\lambda:G\to G$ such that
$$d(\delta_\lambda(p), \delta_\lambda(q) = \lambda d(p,q),\qquad \forall p,q\in G.$$

In this paper we will pursue an Aristotelian approach:  we will begin by discussing the examples, starting from the very basic ones.  We first consider  Abelian Carnot groups, which are the finite-dimensional normed spaces, and then the basic nonAbelian group: the Heisenberg group. After presenting these examples, in Section~\ref{sec:Stratifications} we will formally discuss the definitions from the algebraic viewpoint, with some basic properties. In particular, we show the uniqueness of stratifications in Section~\ref{sec:uniq_strat}.
In Section~\ref{sec:metric} we introduce the homogeneous distances and the general definition of Carnot-Carath\'eodory spaces. In Section~\ref{sec:continuity} we show the continuity of  homogeneous distances with respect to the manifold topology.
Section~\ref{sec:limit} is devoted to present Carnot groups as limits. In particular, we consider tangents of Carnot-Carath\'eodory spaces.
In Section~\ref{sec:Isometrically} we discuss isometrically homogeneous spaces and explain why Carnot-Carath\'eodory spaces are the only geodesic examples.
In Section~\ref{sec:characterization} we provide a metric characterization of Carnot groups as the only spaces that are locally compact geodesic homogeneous and admit a dilation.
Finally, in Section~\ref{SecIsom} we overview several results that provide the regularity of distance-preserving homeomorphisms in Carnot groups and more generally in homogeneous groups and  Carnot-Carath\'eodory spaces.

 %%%%%%%%%%%%%%%%%%%%% %%%%%%%%%%%%%%%%%%%%%
\setcounter{section}{-1}
\section{\textbf{Prototypical examples}} 
We start by reviewing  basic examples of Carnot groups. First we shall point out that  Carnot groups of step one are in fact  
 finite-dimensional normed vector spaces.
Afterwards, we shall recall the simplest non-commutative example: the 
 Heisenberg group, equipped with various distances.

\begin{example}
\label{ex:Abel}
Let $V$ be a  finite dimensional vector space, so $V $ is isomorphic to $ \R^n$, for some $n \in \N$.
In such a space we have
\begin{itemize}
\item a group operation (sum) $p,q \mapsto p+q$,
\item dilations $p \mapsto \lambda p$ for each factor $\lambda >0$.
\end{itemize}
We are interested in the distances $d$ on $V$   that are

(i)   translation invariant:
$$   d(p+q,p+q') = d (q,q')      ,\qquad \forall p,q,q'\in V$$
   
(ii)   one-homogeneous with respect to the dilations:
$$  d(\lambda p, \lambda q) = \lambda d (p,q)    , \qquad \forall p,q \in V,\forall \lambda>0. $$
  
These are the distances coming from norms on $V$.
Indeed, setting $\norm{v}=d(0,v)$ for $v\in V$, the axioms of $d$ being a distance together with properties (i) and (ii), give that $\norm{\cdot}$ is a norm.
\end{example}

A geometric remark: for such distances, straight segments are geodesic.  By definition, a geodesic is an isometric embedding of an interval.
Moreover, if the norm is strictly convex, then straight segments are the only geodesic.

An algebraic remark: each (finite dimensional) vector space can be seen as a Lie algebra (in fact, a commutative Lie algebra) with Lie product $[p,q] = 0 $, for all $p,q\in V$.
Via the obvious identification points/vectors, this Lie algebra is identified with its Lie group $(V,+)$.

One of the theorem that we will generalize in Section~\ref{SecIsom} is the following.
\begin{theorem}  Every isometry of a normed space fixing the origin is a linear map.
\end{theorem}

An analytic remark:  the definition of (directional) derivatives
	$$\lim_{h\to 0} \dfrac{f(x+hy) -  f(x)}{h} = \dfrac{\partial f}{\partial y} (x)$$
of a function $f$ between vector spaces makes use of the group operation, the dilations, and the topology.
We shall consider more general spaces on which these operations are defined.

\begin{example}\label{ex:Heis}

Consider $\R^3$ with the standard topology and differentiable structure.
Consider the operation 
%$$((x,x2,x3),(x1,x2,x3)) \mapsto ...$$
$$(\bar x,\bar y,\bar z) \cdot (x,y,z):= \left(\bar x+x,\bar y+y,\bar z+z+\tfrac{1}{2} (\bar x y-\bar y x)\right).          $$
This operation gives a non-Abelian group structure on $\R^3$.
This group is called {\em Heisenberg group}.

The maps 
$\delta_\lambda (x,y,z) = (\lambda x,\lambda y,\lambda^2 z)$
 give a one-parameter family of group homomorphisms:
$$
	\delta_\lambda(pq) = \delta_\lambda(p) \delta_\lambda(q) \quad \text{ and } \quad
	\delta_\lambda   \circ \delta_\mu = \delta_{\lambda \mu}.
	$$
We shall consider distances that are

(i)	left (translation) invariant:
$$   d(p\cdot q,p\cdot q') = d (q,q')     ,\qquad \forall p,q,q'\in  G   $$
   
(ii)   one-homogeneous with respect to the dilations:
$$  d(\delta_\lambda (p), \delta_\lambda( q)) = \lambda d (p,q)    , \qquad \forall p,q \in G,\forall \lambda>0. $$
%	and one-homogeneous w.r.t. dilations.

These distances, called `homogeneous', are completely characterized by the distance from a point, in  fact, just by the sphere at a point.

\noindent
Example~\ref{ex:Heis}.1. (Box distance) Set
$$d_{box}(0,p):= \norm{p}_{box} :=
\max\left\{|x_p|,|y_p|, \sqrt{|z_p|}\right\}.$$
This function is $\delta_\lambda$-homogeneous and satisfies the triangle  inequality.  
To check that it satisfies the triangle inequality we need to show that 
 $
 \norm{p\cdot q} \leq  \norm{p} + \norm{q}. 
 $ 
 First,
 \[  |x_{p\cdot q}|
=
|x_p +x_q|
\leq 
 |x_p| +|x_q|
 \leq 
 \norm{p} + \norm{q},
 \]
 and analogously for the $y$ component.
 Second,
  \begin{eqnarray*}
 \sqrt{   |z_{p\cdot q}|}
&=&
\sqrt{\left|z_p +z_q +\dfrac{1}{2} (x_p y_q - x_q y_p) \right|}\\
&\leq &
\sqrt{|z_p| +|z_q| +|x_p| |y_q| + |x_q| | y_p |}
\\&\leq& 
\sqrt{  \norm{p}^2 + \norm{q}^2 +2  \norm{p}  \norm{q} } 
%\\&
\qquad  = 
 % &
 \norm{p} + \norm{q}.
 \end{eqnarray*}

The group structure induces left-invariant vector fields, a basis of which is 
$$X=\partial_1-\dfrac{y}{2}\partial_3 , \, \qquad
Y= \partial_2+\dfrac{x}{2}\partial_3, \, \qquad Z= \partial_3. $$

\noindent
Example~\ref{ex:Heis}.2. (Carnot-Carath\'eodory distances)
Fix a norm $\norm{\cdot}$ on $\R^2$, e.g., $\norm{(a,b)} = \sqrt{a^2+b^2}$.
Consider 
\begin{equation*}\label{dist_CC-Heis}
d(p,q):=\inf\left\{   
 \int_0^1 \norm{(a(t),b(t))
  }\mathrm d t
  \right \},
\end{equation*}
where the infimum is over the piece-wise smooth curves
$\gamma\in C_{\rm {pw}}^\infty([0,1];\R^3) $
with $	 \gamma(0)= p$, $
	\gamma(1)= q, $ and
	$\dot\gamma(t) = a(t) X_{\gamma(t)} + b(t) Y_{\gamma(t)}.$ 
Such a $d$ is a homogenous distance and for all $p,q\in \R^3$ there exists  a curve $\gamma$ realizing the infimum. These distances are  examples of CC distances also known as subFinsler distances.

Example~\ref{ex:Heis}.3. (Kor\'anyi  distance)
Another important example is given by the  Cygan-Kor\'anyi distance:
%$$ \|p\|_{g,\alpha}:= \left(\rho(p)^4 + 4 2^2 |z|^2\right)^{1/4}$$
$$d_K(0,p):=\|p\|_{K}:= \left((x_p^2+y_p^2)^2+16\, z_p^2\right)^{1/4} .$$
The feature of such a distance is that it admits a conformal inversion, see \cite[p.27]{Capogna-et-al}.
\end{example}

The space $\R^3$ with the above group structure is an example of Lie group, i.e., the group multiplication and the group inversion are smooth maps.
The space $\g$ of left-invariant vector fields (also known as the Lie algebra) has a peculiar structure: admits a stratification.
Namely, setting
$$V_1:=\text{span}\left\{X,Y \right\} \qquad \text{ and } \quad\qquad V_2:=\text{span}\{Z\},$$
we have
$$\g= V_1\oplus V_2, \quad [V_1, V_1]=V_2, \quad[V_1, V_2]=
\{0\}.$$
As an exercise, verify that $Z=[X,Y]$.

 %%%%%%%%%%%%%%%%%%%%% %%%%%%%%%%%%%%%%%%%%%
\section{\textbf{Stratifications}} %Lie groups, gradings and stratifications}} 
\label{sec:Stratifications}
In this section we discuss
Lie groups,
Lie algebras, 
and their
 stratifications. 
 We recall the definitions and we point out a few remarks. In particular, we show that a group can be stratified in a unique way, up to isomorphism.
\subsection{Definitions}

Given a group $G$ we denote by $gh$ or $g \cdot h$ the product of two elements $g,h\in G$ and by $g^{-1}$ the inverse of $g$.
 A {\em Lie group} is a differentiable manifold   endowed with a group structure such that the   map $
 G\times G\to G,
  (g,h) \mapsto  g^{-1}\cdot h$
 is $C^\infty$. 
 We shall denote by $e$ the
identity of the group and   by
$L_g(h):=g\cdot h$ the left translation.
Any vector $X$ in the tangent space at the identity  extends uniquely to  a left-invariant vector field $\tilde X$, as
$\tilde X_g = (d L_g)_e X$, for $g\in G$.

The {\em Lie algebra associated with a Lie group $G$} is the   vector space
$T_eG$  equipped with the bilinear operation defined by
$[X,Y]:=[\tilde X, \tilde Y]_e$, where the last bracket denotes the Lie bracket of vector fields, i.e., 
$[\tilde X, \tilde Y]:= \tilde X \tilde Y -  \tilde Y \tilde X$.

The general notion of Lie algebra is the following:
 A Lie algebra  $\g$ over $\R$  is a real vector space    together with a bilinear operation 
$$    [\cdot,\cdot]: \mathfrak{g}\times\mathfrak{g}\to\mathfrak{g},$$
called the {\em Lie bracket}, such that, for all $x, y, z \in \mathfrak{g}$, one has
\begin{enumerate}
 \item anti-commutativity: $       [x,y]=-[y,x]$,
 \item Jacobi identity: $[x,[y,z]] + [y,[z,x]] + [z,[x,y]] = 0.$
 \end{enumerate}

All Lie algebras considered here are over $\R$ and finite-dimensional. 
In what follows, given two subspaces $V, W$ of a Lie algebra, we set
$
[V,W] := \Span\{[X,Y] ;\; X\in V,\ Y\in W\} .
$

\begin{Def}[Stratifiable Lie algebras] \label{def:stratifiable-algebras} A \textit{stratification} of a Lie algebra $\g$ is a direct-sum decomposition
$$\g = V_1 \oplus V_2 \oplus \cdots \oplus V_s$$
for some integer $s\geq 1$, where $V_s \not= \{0\}$ and $[V_1,V_j] = V_{j+1}$ for all integers $j\in \{1,\dots,s\}$ and where we set $V_{s+1} = \{0\}$. 
%We call $s$ the step of the stratification. 
We say that a Lie algebra is \textit{stratifiable} if there exists a stratification of it. 
We say that a Lie algebra is \textit{stratified} when it is stratifiable and endowed with a fixed stratification called the {\em associated stratification}.
\end{Def}

A  stratification is a particular example of grading. Hence, for completeness, we proceed by recalling the latter. 
More considerations on this subject can be found in \cite{LeDonne_Rigot_BCP_graded_groups}.

\begin{Def}[Positively graduable Lie algebras]  A \emph{positive grading} of a Lie algebra $\g$ is a
family $(V_t)_{t\in (0,+\infty)}$ of linear subspaces of $\g$, where all but finitely many of the $V_t$'s are $\{0\}$ such that $\g$ is their direct sum
$$\g=  \bigoplus_{t\in (0,+\infty)} V_t $$
  and where 
  %$[V_t, V_u]\subset V_{t+u}$ for all $t,u >0.$
  $$[V_t, V_u]\subset V_{t+u}, \qquad  \text{ for all } t,u >0.$$
We say that a Lie algebra is \emph{positively graduable} if there exists a positive grading of it. 
   We say that a Lie algebra is \textit{graded} (or {\em positively graded}, to be more precise) when it is positively graduable and endowed with a fixed positive grading called the {\em associated positive grading}.
   \end{Def}
 
Given a positive
grading $\g=  \oplus_{t>0} V_t $, the subspace $V_t$ is called the {\em layer of degree} $t$ (or \textit{degree-t layer}) of the positive grading and non-zero elements in $V_t$ are said to {\em have degree $t$}. The {\em degree} of the grading is the maximum of all positive real numbers $t>0$ such that $V_t\neq \{0\}$.

Given two graded Lie algebras $\g$ and $\h$ with associated gradings
  $\g=\oplus_{t>0} V_t$ and $\h=\oplus_{t>0} W_t$,
  a {\em morphism of the graded Lie algebras} is 
  a Lie algebra homomorphism $\phi:\g \rightarrow \h$ such that $\phi (V_t) \subseteq W_t$ for all $t>0$.
  Hence, two graded Lie algebras $\g$ and $\h$ are {\em isomorphic as graded Lie algebras} if there exists a bijection $\phi:\g \rightarrow \h$ such that both $\phi$ and $\phi^{-1}$ are morphism of the graded Lie algebras.

Recall that for a Lie algebra $\g$ the terms of the lower central series are defined inductively by 
$\g^{(1)}=\g$, $\g^{(k+1)}=[\g,\g^{(k)}]$. 
A Lie algebra $\g$ is called {\em nilpotent} if $\g^{(s+1)}=\{0\}$ for some integer $s \geq 1$ and more precisely we say that $\g$ {\em nilpotent of step $s$} if $\g^{(s+1)}=\{0\}$ but $\g^{(s)}\neq \{0\}$.

\begin{remark}
A  positively graduable Lie algebra is 
nilpotent (simple exercise similar to  Lemma~\ref{lem051025}).  
On the other hand, not every nilpotent Lie algebra is positively graduable, see Exercise~\ref{nilpotent not graduable}.  
A stratification of a Lie algebra $\g$
is equivalent to   a  positive grading whose degree-one layer generates $\g$ as a Lie algebra. 
Therefore, given a stratification, the degree-one layer uniquely determines the stratification
and satisfies 
\begin{equation}\label{V1oplus}
\g=V_1\oplus [\g,\g].
\end{equation}
However, 
 an arbitrary vector space $V_1$ that is in direct sum with $[\g,\g]$ (i.e., satisfying \eqref{V1oplus}) may not generate a stratification, see Example~\ref{ex:graded-vs-stratified}.  
Any two stratifications of a Lie algebra are isomorphic,
see Section~\ref{sec:uniq_strat}.  
 A stratifiable Lie algebra with   $s$ non-trivial layers is nilpotent of step $s$.  
Every 2-step nilpotent Lie algebra   is stratifiable.
However, not all graduable Lie algebras   are stratifiable, see Example~\ref{graduable not stratifiable}.
Stratifiable and graduable Lie algebras admit several different grading: given a grading $\oplus V_t$, one can define the so-called $s$-power as the new grading $\oplus W_t$ by setting $W_t=V_{t/s}$, where $s>0$. Moreover, for a stratifiable algebra it is not true that any grading is a power of a stratification, see Example~\ref{ex:Heis:nonstrandard}.
\end{remark}

\begin{example}
Nilpotent-free Lie algebras are stratifiable.  Namely, fixed $r,s\in\N$, one considers formal elements $e_1,\ldots,e_r$ and all possible formal iterated Lie bracket up to length $s$, modulo the anti-commutativity and Jacobi relations, e.g., $[e_1,e_2]$ equals $ -[e_2,e_1]$ and both have length 2. The span of such vectors is the {\em nilpotent-free Lie algebra of rank $r$ and step $s$}. The strata of a stratification of such an algebra are formed according to the length of formal bracket considered.  
Any nilpotent Lie algebra is a quotient of a nilpotent-free Lie algebra, however, the stratification  may not pass to this quotient.
\end{example}

\begin{example}\label{ex:Heis:nonstrandard}
The Heisenberg Lie algebra $\h$ is the 3-dimensional lie algebra spanned by three vectors $X,Y,Z$ and with only nontrivial relation $Z=[X,Y]$, cf Example~\ref{ex:Heis}.
This stratifiable algebra also admits positive gradings that are not power of stratifications. 
In fact, for $\alpha \in (1,+\infty)$, 
%we call \textit{non-standard Heisenberg Lie algebra of exponent $\alpha$} the  Heisenberg Lie algebra equipped with the following 
 the \textit{non-standard grading of exponent} $\alpha$ is
\begin{equation*} %\label{e:nonstandard-grad-heis}
\h =W_1 \oplus W_\alpha \oplus W_{\alpha+1}
\end{equation*}
where $$ W_1:=\Span\{X\} ,\;\; W_\alpha:=\Span\{Y\},\;\; W_{\alpha+1}:=\Span\{Z\}.$$ 
Up to isomorphisms of graded Lie algebras and up to powers, these non-standard gradings give all the possible positive gradings of $\h$ that are not a stratification.
\end{example}

\begin{example}\label{graduable not stratifiable}
Consider the $7$-dimensional Lie algebra  $\g$ generated by $X_1,\ldots, X_7$ with only nontrivial brackets
$$ 
\begin{array} {ccc c ccc c ccc}
{[X_1,X_2]} &=&  X_{ 3} ,&&
{[X_1,X_3]} &=&2 X_{ 4} ,&&
{[X_1,X_4]} &=& 3 X_{ 5} \\
{[X_2,X_3]} &=& X_{ 5} ,&&
{[X_1,X_5]} &=&  4 X_{ 6},& &
{[X_2,X_4]} &=&  2 X_{ 6} \\
{[X_1,X_6]} &=&  5 X_{ 7} ,&&
{[X_2,X_5]} &=&  3 X_{ 7} ,&&
{[X_3,X_4]} &=&X_{ 7}. 
\end{array}$$
This Lie algebra $\g$ admits a grading   but it is not stratifiable.
%\begin{array} {c}
%{[X_1,X_2]} =  X_{ 3} \\
%{[X_1,X_3]} =2 X_{ 4} \\
%{[X_1,X_4]} = 3 X_{ 5} \\
%{[X_2,X_3]} = X_{ 5} \\
%{[X_1,X_5]} =  4 X_{ 6} \\
%{[X_2,X_4]} =  2 X_{ 6} \\
%{[X_1,X_6]} =  5 X_{ 7} \\
%{[X_2,X_5]} =  3 X_{ 7} \\
%{[X_3,X_4]} =X_{ 7}. 
%\end{array}

\end{example}

\begin{example}\label{nilpotent not graduable}
There exist nilpotent Lie algebras that admit no positive grading.
For example, consider the Lie algebra of
dimension 7 with basis $X_1, \ldots , X_7$ with only nontrivial relations given by

$[X_1, X_j] =X_{j+ 1}$ if
$2 \leq j \leq 6$, $\,\,[X_2, X_3] =X_6, \,\,\,[X_2, X_4] = - [X_2, X_5] = [X_3, X_4] =X_7.$
\end{example}

\begin{example} \label{ex:graded-vs-stratified} We give here an example of a stratifiable Lie algebra $\g$ for which one can find a subspace $V$ in direct sum with $[\g,\g]$ but that does not generate a stratification.
We consider $\g$ the stratifiable Lie algebra of step 3 generated by $e_1$, $e_2$ and $e_3$ and with the relation $[e_2,e_3]=0$. Then $\dim \g = 10$ and a stratification of $\g$ is
generated by $ V_1:=\Span\{e_1, e_2, e_3\}$.
Taking $V:=\Span\{e_1, e_2 + [e_1,e_2], e_3\}$, one has \eqref{V1oplus}, but since 
$[V,V]$ and $[V,[V,V]]$ both contain $[e_2,[e_1,e_3]]$, $V$ does not generate a stratification of $\g$. 
\end{example}

\begin{Def}[Positively graduable, graded, stratifiable, stratified groups] \label{def:graded-stratified-groups} We say that a Lie group $G$ is a \emph{positively graduable} (respectively  \emph{graded}, {\em  stratifiable}, \emph{stratified}) \textit{group} if $G$ is a connected and simply connected Lie group whose Lie algebra is positively graduable (respectively graded, stratifiable, stratified). 
\end{Def}

For the sake of completeness, in Theorem~\ref{thm:siebert} below  we present an equivalent definition of positively graduable groups in terms of existence of a contractive group automorphism. The result is due to Siebert.

\begin{Def}[Dilations on graded Lie algebras] \label{def:dilations-algebras}
Let 
$\g$ be a graded Lie algebra with associated positive grading $\g = \oplus_{t>0} V_t $. For $\lambda>0$, 
we define the \emph{dilation on $\g$ (relative to the associated positive grading) of factor $\lambda$} as the unique linear map $\delta_\lambda:\g\to\g$ such that  
$$\delta_\lambda (X) = \lambda^t X \qquad \forall X\in V_t.$$
\end{Def}

Dilations $\delta_\lambda:\g\to\g$ are Lie algebra isomorphisms, i.e., 
	$ \delta_\lambda([X,Y])=[\delta_\lambda X,\delta_\lambda Y]$ for all $X, Y\in \g$. The family of all dilations $(\delta_\lambda)_{\lambda>0}$ is a one-parameter group of Lie algebra isomorphisms, i.e., $\delta_{\lambda}\circ\delta_{\eta}=\delta_{\lambda\eta}$ for all $\lambda, \eta >0$.

\begin{exercise}
Let $\g$ and $\h$ be   graded Lie algebras with associated dilations 
$\delta_\lambda^\g$ and $\delta_\lambda^\h$.
Let $\phi:\g \to \h$ be a Lie algebra homomorphism.
Then $\phi$ is a morphism of graded Lie algebras if and only if
$\phi\circ \delta_\lambda^\g = \delta_\lambda^\h \circ \phi,$ for all $\lambda>0$.
[Solution: Let $\g=  \oplus_{t>0} V_t $ be the grading of $\g$.
If $x\in V_t$ then $\phi (\delta_\lambda x)=\phi (\lambda^t x) = \lambda^t \phi(x)$, which gives the equivalence.]
\end{exercise}

Given a Lie group homomorphism $\phi:G\rightarrow H$, we denote by $\phi_*:\g \rightarrow \mathfrak{h}$ the associated Lie algebra homomorphism. If $G$ is simply connected, given a Lie algebra homomorphism $\psi:\g \rightarrow \mathfrak{h}$, there exists a unique Lie group homomorphism $\phi:G\rightarrow H$ such that $\phi_* = \psi$ (see \cite[Theorem 3.27]{Warner}). This allows to define dilations on $G$ as stated in the following definition. 

\begin{Def}[Dilations on graded groups] \label{def:dilations_groups}
 	Let $G$ be a graded group with Lie algebra $\g$. 
	Let $\delta_\lambda:\g\to\g$ be the dilation on $\g$ (relative to the associated positive grading of $\g$) of factor $\lambda>0$. 
	The \emph{dilation on $G$ (relative to the associated positive grading) of factor $\lambda$} is the unique Lie group automorphism, also denoted by $\delta_\lambda:G\to G$, such that $(\delta_\lambda)_*=\delta_\lambda$.
\end{Def}

For technical simplicity, one keeps the same notation for both dilations on the Lie algebra $\g$ and the group $G$. There will be no ambiguity here. Indeed, graded groups being nilpotent and simply connected, the exponential map $\exp: \g \to G$ is a diffeomorphism from $\g$ to $G$ (see \cite[Theorem~1.2.1]{Corwin-Greenleaf} or \cite[Proposition~1.2]{Folland-Stein}) and one has $\delta_\lambda\circ\exp = \exp\circ\,\delta_\lambda$ (see \cite[Theorem 3.27]{Warner}), hence dilations on $\g$ and dilations on $G$ coincide in exponential coordinates.

For the sake of completeness, we give now an equivalent characterization of graded groups due to Siebert. If $G$ is a topological group and $\tau:G\rightarrow G$ is a group isomorphism, we say that $\tau$ is {\em contractive} if, for all $g\in G$, one has $\lim_{k\to\infty}\tau^k(g)=e$. We say that $G$ is {\em contractible} if $G$ admits a contractive isomorphism.

For graded groups, dilations of factor $\lambda <1$ are contractive isomorphisms, hence positively graduable groups are contractible. Conversely, Siebert proved (see Theorem~\ref{thm:siebert} below) that if $G$ is a connected locally compact group and $\tau:G\rightarrow G$ is a contractive isomorphism then $G$ is a connected and simply connected Lie group and $\tau$ induces a positive grading on the Lie algebra $\g$ of $G$ (note however that $\tau$ itself  may not be a dilation relative to the induced grading).

\begin{thm} \label{thm:siebert} \cite[Corollary~2.4]{Siebert} A topological group $G$ is a positively graduable Lie group if and only if $G$ is a connected locally compact contractible group.
\end{thm}

\begin{proof}[Sketch of the proof]
Regarding the nontrivial direction, by the general theory of locally compact groups one has that $G$ is a Lie group. Hence, the 
contractive group isomorphism induces a contractive Lie algebra isomorphism $\phi$. 
Passing to the complexified Lie algebra, one consider the Jordan form of $\phi$ with generalized eigenspaces $V_\alpha$, $\alpha\in \C$.
The $t$-layer $V_t$ of the grading is then defined as
 the real part of the span of those $ V_\alpha$ with $-\log |\alpha| = t$.
\end{proof}

\begin{remark}
A distinguished class of groups in Riemannian geometry are the so-called Heintze's groups. They are those groups that admit a structure of negative curvature.
It is possible to show that  such groups are precisely the direct product of a graded group $N$ times $\R$ where $\R$ acts on $N$ via the grading.
\end{remark}

\subsection{Uniqueness of stratifications}\label{sec:uniq_strat}
We start by observing the following simple fact.
\begin{lemma}\label{lem051025}
 	If $\g=V_1\oplus\dots\oplus V_s$ is a stratified Lie algebra, then
	\[
	\g^{(k)} = V_k\oplus\dots\oplus V_s .
	\]
	In particular, $\g$ is nilpotent of step $s$.
\end{lemma}

Now we show that the stratification of a stratifiable Lie algebra is unique up to isomorphism.
Hence, also the structure of a stratified group is essentially unique.
\begin{proposition}
 	Let $\g$ be a stratifiable Lie algebra with two stratifications,
	\[
	V_1\oplus\dots\oplus V_s = \g = W_1\oplus\dots\oplus W_t .
	\]
	Then $s=t$ and 	there is a Lie algebra automorphism $A:\g\to\g$ such that $A(V_i)=W_i$ for all $i$.
\end{proposition}
\begin{proof} 	
	We have $\g^{(k)}=V_k\oplus\dots\oplus V_s = W_k\oplus\dots\oplus W_t$.
	Then
	$s=t$. Moreover,
	 the quotient mappings $\pi_k:\g^{(k)}\to \g^{(k)}/\g^{(k+1)}$ induce linear isomorphisms $\pi_k|_{V_k}:V_k\to \g^{(k)}/\g^{(k+1)}$ and $\pi_k|_{W_k}:W_k\to \g^{(k)}/\g^{(k+1)}$.
	For $v\in V_k$ define $A(v):= (\pi_k|_{W_k})^{-1}\circ \pi_k|_{V_k}(v)$. 
	Explicitly, for $v\in V_k$ and $w\in W_k$ we have 
	\[
	A(v)=w \quad\iff\quad v-w\in \g^{(k+1)}.
	\]
	Extend $A$ to a linear map $A:\g\to\g$. 
	This is clearly a linear isomorphism and $A(V_i)=W_i$ for all $i$.
	We need now to show that $A$ is a Lie algebra morphism, i.e., $[Aa,Ab]=A([a,b])$ for all $a,b\in \g$.	
	Let $a=\sum_{i=1}^sa_i$ and $b=\sum_{i=1}^sb_i$ with $a_i,b_i\in V_i$.
	Then
	\begin{align*}
	 	A([a,b]) &= \sum_{i=1}^s\sum_{j=1}^s A([a_i,b_j]) \\
		[Aa,Ab] &= \sum_{i=1}^s\sum_{j=1}^s [Aa_i,Ab_j] ,
	\end{align*}
	therefore we can just prove $A([a_i,b_j])=[Aa_i,Ab_j]$ for $a_i\in V_i$ and $b_j\in W_j$.
	Notice that $[a_i,b_j]$  belongs to $V_{i+j}$ and $[Aa_i,Ab_j]$   belongs to $W_{i+j}$. 
	Therefore we have $A([a_i,b_j])=[Aa_i,Ab_j]$ if and only if  
	$[a_i,b_j] - [Aa_i,Ab_j]\in\g^{(i+j+1)}$.
	And in fact
	\[
	[a_i,b_j] - [Aa_i,Ab_j] 
	= [a_i-Aa_i,b_j] - [ Aa_i,Ab_j-b_j] \in \g^{(i+j+1)}
	\]
	because, on the one hand, $a_i-Aa_i\in \g^{(i+1)}$ and $b_j\in W_j$, so $[a_i-Aa_i,b_j]\in\g^{(i+j+1)}$, on the other hand, $Aa_i\in W_i$ and $Ab_j-b_j\in \g^{(j+1)}$, so $[ Aa_i,Ab_j-b_j] \in \g^{(i+j+1)}$. 
\end{proof}

 %%%%%%%%%%%%%%%%%%%%% %%%%%%%%%%%%%%%%%%%%%
\section{\textbf{Metric groups}}\label{sec:metric}

In this section we review homogeneous distances on  groups. 
The term {\em homogeneous} should not be confused with its use in the theory of Lie groups. Indeed, clearly a Lie group is a homogeneous space since it acts on itself by left translations and in fact we will only consider distances that are left-invariant.
We shall use the term homogeneous as it is done in harmonic analysis to mean that there exists a transformation that dilates the distance.
In the literature there is another definition of {\em homogeneous group}, but even if this definition is algebraic, it coincides with our definition. Namely, around 1970 Stein introduced a definition of homogeneous group as a graded group with degrees greater or equal than 1. 
After \cite{Hebisch-Sikora}, we know that these are exactly those groups groups that admit a homogeneous distance, which is defined as follows.

\begin{definition}[Homogeneous distance]

Let $G$ be a Lie group.
Let $\{\delta_\lambda\}_{\lambda\in (0,\infty)}$ be a one-parameter family of continuous automorphisms of $G$.
A distance $d$ on $G$ is called {\em homogeneous} if
\begin{enumerate}
%\item[(i)]
%  it is {\em admissible}, i.e., it induces the manifold topology;
\item[(i)] it is  {\em left-invariant}, i.e., 
$$   d(p\cdot q,p\cdot q') = d (q,q') , \qquad \forall p,q,q'\in G;$$
   \item[(ii)]  it is  {\em one-homogeneous} with respect to $\delta_\lambda$, i.e., 
$$  d(\delta_\lambda (p), \delta_\lambda (q) ) = \lambda d (p,q),\forall p,q\in G, \forall \lambda>0. $$
\end{enumerate}
\end{definition}

Some remarks are in due:

(1) 
%Condition (i) is actually consequence of (ii) and (iii)
 Conditions (i) and (ii) implies that the distance is
 {\em admissible}, i.e., it induces the manifold topology, and in fact $d$ is a continuous function\footnote{This claim is not at all trivial. However, in Section~\ref{sec:continuity} we provide the proof in the case the group is positively graded and the maps $\{\delta_\lambda\}$ are the dilations relative to the grading.};
  
(2) Because of the existence of one such a distance, the automorphisms $\delta_\lambda$, with $\lambda\in (0,1)$, are contractive.

(3) The existence of a contractive automorphism  implies the existence of a grading of the Lie algebra of the group, see Theorem~\ref{thm:siebert}.

(4) The existence of a homogeneous distance implies that the smallest degree of the grading should be $\geq 1$.
These kind of groups are called {\em homogeneous} by Stein and collaborators.

(5) Any homogeneous group admits a homogeneous distance, see \cite{Hebisch-Sikora} and \cite[Section 2.5]{LeDonne_Rigot_BCP_graded_groups}

\subsection{Abelian groups}
The basic example of a homogeneous group is provided by Abelian groups, which 
admit the stratification where the one-degree stratum $V_1$ is the whole Lie algebra.
Homogenous distances for this stratification are the distances induced by norms, cf.~Section~\ref{ex:Abel}.

%1) Finite dimensional normed vector spaces
 
\subsection{Heisenberg group}
 Heisenberg group equipped with CC-distance, or box distance, or Kor\'anyi distance
 is the next important example of homogeneous group, cf.~Section~\ref{ex:Heis}.  

\subsection{Carnot groups}
\label{sec:Carnot}
Since on stratified groups the degree-one stratum $V_1$ of the stratification of the Lie algebra generates the whole Lie algebra, on these groups there are homogeneous distances that have a length structure defined as follows.

Let $G$ be a stratified group. So $G$ is simply connected Lie group and its Lie algebra $\g$ has a stratification: $\g = V_1 \oplus V_2 \oplus \cdots \oplus V_s$.
Therefore, we have $\g = V_1 \oplus [V_1,V_1] \oplus \cdots \oplus V_1^{(s)}$.

Fix a norm $\norm{\cdot}$ on $V_1$.
Identify the space $\g$ as $T_eG$ so 
$V_1\subseteq T_e G$. 
By left translation, extend $V_1$ and $\norm{\cdot}$ to a left-invariant subbundle $\Delta$ and a left-invariant norm $\norm{\cdot}$:
$$ \Delta_p := ({\rm d} L_p)_e V_1 \qquad\text{ and }\qquad \norm{({\rm d} L_p)_e(v)} := \norm{v},\qquad \forall p\in G, \forall v\in V_1.$$
Using piece-wise $C^\infty$ curves tangent to $\Delta$, we define the {\em CC-distance} associated with $\Delta$  and $\norm{\cdot}$ as
\begin{equation}\label{dist_CC-Carnot}
d(p,q):=\inf\left\{   \left.
    \int_0^1 \norm{\dot\gamma(t)  }\mathrm d t
    \;\right|\;\gamma\in C_{\rm {pw}}^\infty([0,1];G), 
    \begin{array}{ccc}
	 \gamma(0)= p, \\
	\gamma(1)= q, 
	\end{array}
%    \text{ and }
   \dot \gamma \in \Delta
   \right \}.
\end{equation}

Since $\Delta$  and $\norm{\cdot}$ are left-invariant, $d$ is left-invariant.
Since, 
$\delta_\lambda V_1 = \lambda V_1$, $\delta_\lambda \Delta = \lambda \Delta$,  and $\norm{\lambda v}=\lambda \norm{ v}$,
$d$ is one-homogeneous with respect to $\delta_\lambda$.

The fact that $V_1$ generates $\g$, implies that
for all $p,q\in G$ there exists
 a curve 
$ \gamma\in C^\infty([0,1];G) $
with
  $ \dot \gamma \in \Delta
  $  joining $p$ to $q$. Consequently, $d$ is finite valued.

We call the data $(G, \delta_\lambda, \Delta,\norm{\cdot}, d )$ a {\em Carnot group}, or, more explicitly, {\em 
subFinsler Carnot group}. Usually, the term Carnot group is reserved for  {\em 
subRiemannian Carnot group}, i.e., when the norm comes from a scalar product.

\subsection{Carnot-Carath\'eodory spaces}

Carnot groups are particular examples of a more general class of spaces.
These spaces have been named after Carnot and Carath\'eodory by Gromov. However, in the work of  Carnot and Carath\'eodory there is very little about these kind of geometries (one can find some sprout of a notion of contact structure in \cite{Caratheodory} and in the adiabatic processes' formulation of Carnot).  
 Therefore, we should say that the pioneer work has been done mostly by Gromov and
Pansu, see 
\cite{PansuThesis,  Gromov-polygrowth, pansucras, Pansu83, Pansu-croissance, Strichartz, Pansu, Hamenstadt}, see also \cite{Gromov1, Gromov}.

Let $M$ be a smooth manifold.
Let $\Delta$ be a subbundle of the tangent bundle of $M$.
Let $\norm{\cdot}$ be a `smoothly varying' norm on $\Delta$.
Analogously,
using 
   \eqref{dist_CC-Carnot} with $G$ replaced by $M$,
   one defines
    the {\em CC-distance} associated with $\Delta$  and $\norm{\cdot}$.     
Then $(M, \Delta,\norm{\cdot}, d )$ is called {\em Carnot-Carath\'eodory space} (or also {\em  CC-space} or {\em  subFinsler manifold}).

\subsection{Continuity of homogeneous distances} \label{sec:continuity}
We want to motivate now the fact that a homogeneous distance on a group induces the correct topology.
Such a fact is also true when one considers quasi distances and dilations that are not necessarily $\R$-diagonalizable automorphisms.
%, see \cite{Enrico&Seba}.
For the sake of simplicity, we present here the simpler case of distances that are homogeneous with respect to the standard dilations.
The argument is taken from \cite{LeDonne_Rigot_BCP_graded_groups}. The complete proof will appear in \cite{Enrico&Seba}. 

\begin{proposition}\label{homogeneous distance topology}
Every homogeneous distance  induces the manifold topology.
\end{proposition}

\begin{proof}[Proof in the case of standard dilations]
Let $G$ be a  group with identity element $e$ and Lie algebra graded by $\g=\oplus_{s>0} V_s $. Let $d$ be a distance homogeneous with respect to
the standard dilations.
It is enough to show that the topology induced by $d$ and the manifold topology give the same neighborhoods at the identity.

We show that if $p$ converges to $e$ then $d(e,p)\to 0$.
Since $G$ is 
nilpotent, we can consider
exponential coordinates of second kind with respect to a basis $X_1, 
\ldots, X_n$ of $\g$
adapted to the grading.
 Namely, first 
for all $i=1,\ldots, n$ there is
$d_i>0$ such that
 $X_i\in V_{d_i}$, and consequently, $ \delta_\lambda (X)= \lambda^{d_i} X$.
There is a diffeomorphism $p \mapsto (P_1(p) ,  \ldots,   P_n(p) )$  from $G$ to $\R^n$ such that
for all $p\in G$
$$p= \exp (P_1(p) X_1)  \cdot\ldots\cdot \exp (P_n(p) X_n).$$
Notice that $P_i(e)=0$.
Then, using the  triangle inequality, the left invariance, and the homogeneity of the distance, we get
\begin{equation*}
d(e,p)\leq   \sum_i   d(e, \exp (P_i(p) X_i) )  
%&= &  \sum_i  d(e, \exp ( \delta_{P_i(p)^{1/d_i}   } X_i) ) \\
=   \sum_i   | P_i(p) | ^{1/d_i}d(e, \exp (X_i) ) \to 0, \quad \text{ as } p\to e. 
\end{equation*} 

We show that if  $d(e,p_n)\to 0$ then $p_n$ converges to $e$. 
By contradiction and up to passing to a subsequence, there exists $\eps>0$ such that 
$\norm{p_n} >\eps$, where $\norm{\cdot}$ denotes any auxiliary Euclidean norm.
Since the map $\lambda\mapsto \norm{\delta_\lambda(q)}$ is continuous, for all $q\in G$, then for all $n$ there exists $\lambda_n \in (0,1)$ such that 
$\norm{ \delta_{\lambda_n}(p_n)} =\eps$.
Since the Euclidean $\eps$-sphere is compact, up to subsequence, 
$ \delta_{\lambda_n}(p_n)\to q$, with $\norm{q}=\eps$ so $q\neq e$ and so $d(e, q)>0.$
However,
%\begin{eqnarray*}
$$0< d(e, q)\leq   d(e, \delta_{\lambda_n}(p_n))  +  d(\delta_{\lambda_n}(p_n), q) =   \lambda_n d(e,  p_n)  +  d(\delta_{\lambda_n}(p_n), q) 
\to 0,$$
%\end{eqnarray*}
where at the end we used the fact, proved in the first part of this proof, that if $q_n$ converges to $q$ then $d(q_n,q)\to 0$. 
\end{proof}

\begin{remark}
To deduce that a homogeneous distance is continuous is necessary to require that the automorphisms used in the definition are continuous with respect to the manifold topology. Indeed otherwise, consider the following example.
Via a group isomorphism from $\R$ to $\R^2$,   pull back the standard dilations and the Euclidean metric from $\R^2$ to $
\R$. In this way we get a one-parameter family of dilations and a homogeneous distance on $\R^2$.
Clearly, this distance gives to $\R^2$ the topology of the standard $\R$.
\end{remark}

 %%%%%%%%%%%%%%%%%%%%% %%%%%%%%%%%%%%%%%%%%%
\section{\textbf{Limits of Riemannian manifolds}}\label{sec:limit}

\subsubsection{A topology on the space of metric spaces}\label{sec:GH_top}

Let $X$ and $Y$ be metric spaces, $L>1$ and $C>0$.
A map $\phi:X\to Y$ is an \emph{$(L,C)$-quasi-isometric embedding} if for all $x,x'\in X$
\[
\frac1L d(x,x') - C
\le d(\phi(x),\phi(x')) \le 
L d(x,x') + C .
\]
If $A,B\subset Y$ are subsets of a metric space $Y$ and $\epsilon>0$, we say that $A$ is an \emph{$\epsilon$-net} for $B$ if 
\[
B\subset {\rm Nbhd}_\epsilon^Y(A):=\{ y\in Y:d(x,A) < \epsilon \}.
\]

\begin{Def}[Hausdorff approximating sequence]
 	Let $(X_j,x_j),(Y_j,y_j)$ be two sequences of pointed metric spaces.
	A sequence of maps $\phi_j:X_j\to Y_j$ with $\phi(x_j)=y_j$ is said to be \emph{Hausdorff approximating} if for all $R>0$ and all $\delta>0$ there exists $\epsilon_j $ such that
	\begin{enumerate}
	\item 	$\epsilon_j\to 0$ as $j\to \infty$;
	\item 	$\phi_j|_{B(x_i,R)}$ is a $(1,\epsilon_j)$-quasi isometric embedding;
	\item 	$\phi_j(B(x_j,R))$ is an $\epsilon_j$-net for $B(y_j,R-\delta)$.
	\end{enumerate}
\end{Def}

\begin{Def}
 	We say that a sequence of pointed metric spaces $(X_j,x_j)$ converges to a pointed metric space $(Y,y)$ if there exists an Hausdorff approximating sequence $\phi_j:(X_j,x_j)\to (Y,y)$.
\end{Def}
This notion of convergence was introduced by  Gromov \cite{Gromov-polygrowth} and it is also called \emph{Gromov-Hausdorff convergence}.
It defines a topology on the collection of (locally compact, pointed) metric spaces, which extends the notion of
uniform convergence on compact sets of distances on the same topological space.

\begin{Def}If $X=(X,d)$ is a metric space and $\lambda>0$, we set $\lambda X = (X,\lambda d)$.

 	Let $X,Y$ be  metric spaces, $x\in X$ and $y\in Y$.
	
	We say that $(Y,y)$ is the {\em asymptotic cone} of $X$ if  for all $\lambda_j\to0$, $(\lambda_j X,x)\to (Y,y)$.
	
	We say that $(Y,y)$ is the {\em tangent space} of $X$ at $x$ if for all $\lambda_j\to\infty$, $(\lambda_j X,x)\to (Y,y)$.
\end{Def}
\begin{Rem}
 	The notion of asymptotic cone is independent from $x$.
 	In general, asymptotic cones and tangent spaces may not exist.
 	In the space of boundedly compact metric spaces, limits are unique up to isometries.
\end{Rem}

\subsection{Tangents to CC-spaces: Mitchell Theorem}

The following result states that Carnot groups are infinitesimal models of CC-spaces:
the tangent metric space to an equiregular subFinsler manifold is a subFinsler Carnot group.
We recall that a subbundle $\Delta \subseteq TM$ is {\em equiregular} if for all $k\in \N$
$$\bigcup_{q\in M} 
{\rm span}\{[Y_1,[Y_2,[\dots [Y_{l-1},Y_l]]]]_q\;:\; l\leq k,Y_j\in\Gamma(\Delta),   j=1,\dots,l\}
$$
defines a subbundle of $TM$.
For example, on a Lie group $G$ any $G$-invariant subbundle is equiregular.

For the next theorem see \cite{Mitchell, bellaiche, %} and \cite{
Margulis-Mostow, Margulis-Mostow2, jeancontrol}. %,Enrico&Seba&Woocel}.

\begin{theorem}[Mitchell]\label{thm:Mitchell}
 	Let $M$ be an equiregular subFinsler manifold and $p\in M$.
	Then the tangent space of $M$ at $p$ exists and is a subFinsler Carnot group.
\end{theorem}

\subsection{Asymptotic cones of nilmanifolds}
We point out another theorem showing how naturally Carnot groups appear in Geometric Group Theory.
The result is due to Pansu, and proofs can be found in \cite{Pansu-croissance, Breuillard-LeDonne1,Breuillard-LeDonne1_1}.
\begin{theorem}[Pansu]
 	Let $G$ be a nilpotent Lie group equipped with a left-invariant subFinsler distance.
	Then the asymptotic cone of $G$ exists and is a subFinsler Carnot group.
\end{theorem}

%\subsection{A proof of Pansu Theorem by example}

 %%%%%%%%%%%%%%%%%%%%% %%%%%%%%%%%%%%%%%%%%%
\section{\textbf{Isometrically homogeneous geodesic manifolds}}
\label{sec:Isometrically}

\subsection{Homogeneous metric spaces}

Let $X=(X,d)$ be a metric space.
Let ${\rm Iso}(X)$ be the 
group of self-isometries of $X$, i.e., distance-preserving homeomorphisms of $X$.

\begin{definition}\label{def:isometrically}
A metric space 
$X$ is {\em isometrically homogeneous} (or is a {\em homogeneous metric space}) if
for all $x,x'\in X$ there exists $f\in {\rm Iso}(X)$
 with $f(x)=x'$.
\end{definition}

The main examples are the following. Let $G$ be Lie group and $H<G$ compact subgroup.
So the collection of the right cosets $G/H:=\{ gH : g\in G\}$ is an analytic manifold on which $G$ acts (on the left) by analytic maps.
Since $H$ is compact, it always exists a $G$-invariant distance $d$ on $G/H$.

Let us present some regularity results for isometries of such spaces. Later, in Section~\ref{SecIsom1} we shall give more details.

\begin{theorem}[LD,  Ottazzi]
The isometries of the above examples are analytic maps.
\end{theorem}
The above result is due to the author and Ottazzi, see \cite{LeDonne-Ottazzi}.
The argument is based on the structure theory of locally compact groups, see next two sections. 
Later, in Theorem~\ref{compactness}, we will discuss a  stronger statement.

We present now an immediate consequence saying that in these examples the isometries  are also Riemannian isometries for some Riemannian structure.
In general, the Riemannian isometry group may be larger, e.g., in the case of $\R^2$ with $\ell_1$-norm.
\begin{corollary}
For each   $G$-invariant distance $d$ on $G/H$ there is a Riemannian  $G$-invariant distance $d_R$ on $G/H$ such that
 $${\rm Iso}(G/H,d) <   {\rm Iso}(G/H,d_R).$$
\end{corollary}

\proof[Sketch of the proof of the corollary]

Fix $p \in G/H$ and $\rho_p$ a scalar product on
$T_p(G/H)$.
Let
$K$ be the stabilizer of $p$ in 
${\rm Iso}(G/H,d)$, which is compact by Ascoli-Arzel\`a.
Let 
$\mu$ be a Haar measure on $K$.
Set
$$\tilde \rho_p=\int_K F^*  \rho_p  \; \mathrm d \mu(F)$$
 and
 for all $q\in G/H$ set $\tilde \rho_q= F^*\tilde \rho_p$, 
  for some $F\in  {\rm Iso}(G/H,d)$ with
$F(q)=p$.
Then $\tilde \rho$ is a Riemannian metric tensor that 
is ${\rm Iso}(G/H,d)$-invariant.
\qed

\subsection{Berestovskii's characterization}

Carnot groups are examples of isometrically homogeneous spaces. In addition, their CC-distances have the property of having been constructed by a length structure.
Hence, these distances are intrinsic in the sense of the following definition. For more information on length structures and intrinsic distances, see \cite{Burago:book}. 

\begin{definition}[Intrinsic distance]\label{def:intrinsic}
A distance $d$ on a set $X$ is {\em intrinsic} if for all $x,x'\in X$ 
$$d(x,x') = \inf \{ L_d(\gamma)\},$$
where the infimum is over all curves $\gamma \in C^0([0,1];X)$ from $x$ to $x'$ and  $$L_d(\gamma):=\sup \sum d(\gamma(t_i),\gamma(t_{i-1}) ),$$ where the supremum is over all partitions $t_0< \ldots< t_k$ of $[0,1]$.  
\end{definition}

Examples of intrinsic distances are given by length structures,
subRiemannian structures,
Carnot-Carath\'eodory spaces, and 
geodesic spaces. A metric space whose distance is intrinsic is called {\em geodesic} if the infimum in   Definition~\ref{def:intrinsic} is attained.

The significance of the next result is that 
CC-spaces are natural objects in the theory of homogeneous metric spaces. 
\begin{theorem}[Berestovskii, \cite{b}] \label{thm:Berestovskii}
If a homogeneous Lie space $G/H$ is equipped with a $G$-invariant intrinsic distance $d$,
then
$d$ is subFinsler, i.e., there exist a $G$-invariant subbundle $\Delta$
and a $G$-invariant norm $\norm{\cdot}$ such that  $d$ is the CC distance associated with $\Delta$
and $\norm{\cdot}$.
\end{theorem}

\proof[How to prove Berestovskii's result]
In three steps:
\begin{itemize}
\item[Step 1.] Show that locally $d$ is $\geq$ to some Riemannian distance $d_R$.

\item[Step 2.] Deduce that  curves that have finite length with respect to $d$ are Euclidean rectifiable and define the horizontal bundle $\Delta$ using velocities of  such curves. Definite $\norm{\cdot}$ similarly.

\item[Step 3.] Conclude that $d$ is the CC distance for $\Delta$
and $\norm{\cdot}$.
 \qed
\end{itemize}

The core of the argument is in Step 1. Hence we describe its proof in a exemplary case.
\proof[Simplified proof of Step 1]
We only consider the following simplification: $G=G/H=\R^2$, i.e.,  $d$ is a translation invariant distance on the plane.
Let $d_E$ be the Euclidean distance.
We want to show that $d_E/d$ is locally bounded.
If not, there exists $p_n\to0$ such that
$\dfrac{d_E(p_n,0)}{d(p_n,0)} >n$.
Since the
topologies are the same, there exists $r>0$ such that
$B_d(0,r) \Subset B_E(0,1) $.
Since $p_n\to0$, there is $h_n\in \N$ such that eventually
$h_np_n\in B_E(0,1) \setminus B_d(0,r)$.
Hence,
\begin{eqnarray*}
0<r<d(h_n p_n,0) = h_n d(p_n,0) &\leq& \dfrac{h_n}{n} d_E(p_n,0)\\
&=&\dfrac{1}{n} d_E(h_np_n,0)\leq \dfrac{1}{n} \to 0,
\end{eqnarray*}
which gives a contradiction.\qed

 %%%%%%%%%%%%%%%%%%%%% %%%%%%%%%%%%%%%%%%%%%
 \section{\textbf{A metric characterization of Carnot groups}}
 \label{sec:characterization}
\subsection{Characterizations of Lie groups}  

Providing characterization of Lie groups among topological groups was one of the 
Hilbert problems: the 5$^{th}$ one.
Nowadays, it is considered solved in various forms by the work of 
John von Neumann,
   Lev Pontryagin,
    Andrew Gleason, Deane Montgomery, Leo Zippin,
and Hidehiko Yamabe. See \cite{mz} and references therein. For the purpose of characterizing Carnot groups among homogeneous metric spaces, we shall make use of the following.

\begin{theorem}[Gleason - Montgomery - Zippin. 1950's]\label{thm:GMZ}
Let $X$ be a metric space that is connected, locally connected, locally compact, of finite topological dimension, and isometrically homogeneous. Then its isometry group ${\rm Iso}(X)$ is a Lie group.
\end{theorem} 
As a consequence, the isometrically homogeneous metric spaces considered in 
Theorem~\ref{thm:GMZ} are all of the form discussed 
in Section~\ref{sec:Isometrically}
after Definition~\ref{def:isometrically}.

\proof[How to prove  Theorem \ref{thm:GMZ}]
In many steps:

\begin{itemize}
\item[Step 1.]  ${\rm Iso}(X)$ is locally compact, by Ascoli-Arzel\'a.

\item[Step 2.]  Main Approximation Theorem:
There exists an open and closed subgroup $G$ of ${\rm Iso}(X)$ that can be approximated by Lie groups:
$$G= \varprojlim G_i,\quad \text{ (inverse limit of continuous epimorphisms with compact kernel)}$$
by Peter-Weyl and Gleason.

\item[Step 3.] $G\acts X$  transitively, by Baire; so
$X=G/H = \varprojlim  G_i/H_i$.

\item[Step 4.]  Since the topological dimension of $X$ is finite and $X$ is locally connected, for $i$ large 
$G_i/H_i \to G/H$ is a homeomorphism.

\item[Step 5.] $G=G_i$  for $i$ large, so $G$ is a Lie group so $G$ is NSS, i.e., $G$ has no small subgroups.

\item[Step 6.]  $G$ is NSS so ${\rm Iso}(X)$ is NSS

\item[Step 7.]  Locally compact groups with NSS are Lie groups, by Gleason.
\qed
\end{itemize}

\subsection{A metric characterization of Carnot groups} 

In what follows, we  say that  a metric space $(X,d)$ is {\em  self-similar} if there exists $\lambda>1$ 
such that the metric space $(X,d)$ is isometric to the metric space $(X,\lambda d)$.
In other words, there exists 
 a homeomorphism  $f:X\to X$ such that $$d(f(p), f(q) )=\lambda d(p,q) ,\text{  for all }p,q\in X.$$
When this happens for all  $\lambda>0$ (and all maps $f=f_\lambda$ fix a point) $X$ is said to be a {\em cone}.
Homogeneous groups are examples of cones.

The following result is a corollary  of the work of  Gleason-Montgomery-Zippin, Berestovskii, and Mitchell, see \cite{LeDonne_characterization}. It gives a metric characterization of Carnot groups.
 \begin{theorem}\label{characterization_Carnot}
SubFinsler Carnot groups are the only metric spaces that are  
\begin{enumerate}
\item  locally compact, 
\item  geodesic, 
  \item isometrically homogeneous, and
   \item  self-similar.
\end{enumerate}
 \end{theorem}

\proof[Sketch of the proof]
Each such metric space $X$ is connected and locally connected.
Using the conditions of local compactness, 
  self-similarity, and homogeneity, one can show that $X$ is a doubling metric space.
In particular, $X$ is finite dimensional.
By the result of Gleason-Montgomery-Zippin (Theorem~\ref{thm:GMZ}) the space  $X$ has the structure of a homogeneous Lie space   $G/H$ and, by Berestovskii's result (Theorem~\ref{thm:Berestovskii}), as a metric space $X$ is an equiregular subFinsler manifold.
By Mitchell's result (Theorem~\ref{thm:Mitchell}), the tangents of $X$ are subFinsler Carnot groups.
Since $X$ is self-similar, $X$ is isometric to its tangents.
\qed

 %%%%%%%%%%%%%%%%%%%%% %%%%%%%%%%%%%%%%%%%%%
\section{\textbf{Isometries of metric groups}} \label{SecIsom}

\subsection{Regularity of isometries for homogeneous spaces}\label{SecIsom1}

Let $X$ be a metric space that is connected, locally connected, locally compact, with finite topological dimension, and isometrically homogeneous.
By Montgomery-Zippin,  $G:={\rm Iso} (X)$ has the structure of (analytic)  Lie group, which is unique.
Fixing $x_0\in X$, the subgroup $H:={\rm Stab}_G(x_0)$ is compact.
Therefore, $G/H$ has an induced structure of analytic manifolds.
We have that $X$ is homeomorphic to $G/H$ and $G$ acts on $G/H$ analytically.
One may wonder if $X$ could have had a different differentiable structure.
Next result says that this cannot happen.

\begin{theorem}[LD,  Ottazzi,  \cite{LeDonne-Ottazzi}]  
\label{compactness}
Let $M=G/H$ be a homogeneous manifold  
equipped with a   $G$-invariant  
 distance  $d$, inducing the manifold topology.
Then the isometry group ${\rm Iso} (M)$ is a Lie group, 
the action
 \begin{eqnarray}\label{Iso_azione}
 {\rm Iso} (M)\times M&\to& M\\
 (F,p)&\mapsto& F(p) \nonumber
 \end{eqnarray}
 is analytic, and, for all $p\in M$, the space 
 $
{\rm Iso}_p (M) 
$
is a compact Lie group.
\end{theorem}

The above result is just a consequence of Theorem~\ref{thm:GMZ} and the   
uniqueness of analytic structures for homogeneous Lie spaces, see \cite[Proposition 4.5]{LeDonne-Ottazzi}.

\subsection{Isometries of Carnot groups}

We want to explain in further details what are the isometries of Carnot groups.
We summarize our knowledge by the following results.

\begin{itemize}
\item \cite{mz} implies that the global isometries are smooth, being Carnot groups examples of homogeneous spaces.
\item \cite{Capogna-Cowling} implies that isometries between open subsets of Carnot groups are smooth, since 1-quasi-conformal maps are.
\item  \cite{Capogna_LeDonne} implies that isometries between equiregular subRiemannian manifolds are smooth.
\item \cite{LeDonne-Ottazzi} implies that isometries between open subsets of subFinsler Carnot groups are affine, i.e., composition of translations and group homomorphisms.
\item  \cite{Kivioja_LeDonne_isom_nilpotent}   implies that isometries of nilpotent connected Lie groups are affine.
\end{itemize}
In the rest of this exposition we  give more explanation on   the last three points.

\subsection{Local isometries of Carnot groups} 

\begin{theorem}[LD,  Ottazzi,  
\cite{LeDonne-Ottazzi}]\label{localisometries}
Let $G_1, G_2$ be subFinsler Carnot groups and for $i=1,2$ consider  $\Omega_i\subset G_i$  open sets. If $F:\Omega_1\to \Omega_2$ is an isometry, then 
there exists a left translation $\tau$ on $G_2$ and a  group isomorphism $\phi $ between $G_1$ and $G_2$, 
such that $F$ is the restriction to $\Omega_1$ of $\tau\circ\phi$, which is a global isometry.
\end{theorem}

\proof[Sketch of the proof]
In four steps:
\begin{itemize}
\item[Step 1.] We may assume that the distance is subRiemannian, regularizing the subFinsler norm.
\item[Step 2.] $F$ is smooth; this is a PDE argument using the regularity of the subLaplacian, see next session. Alternatively, one can use \cite{Capogna-Cowling}.
\item[Step 3.] $F$ is completely determined by its horizontal differential $(\mathrm d F)_{H_eG}$, see the following Corollary~\ref{horizontal:differential}. 
\item[Step 4.] The Pansu differential $(PF)_e$ exists and is an isometry with same  horizontal differential of $F$ at $e$, see \cite{Pansu}.
 \qed
\end{itemize}

We remark that in Theorem~\ref{localisometries} the assumption that $\Omega_i$ are open is necessary, unlike in the Euclidean case. However, these open sets are not required to be connected.
\subsection{Isometries of subRiemannian manifolds}
The regularity of subRiemannian isometries should be thought as into two steps, where as an intermediate result one obtains the preservation of a good measure: the Popp measure.
A good introduction to the notion of Popp measure can be found in \cite{Barilari-Rizzi}.
Both of next results are due to the author in collaboration with Capogna, see  \cite{Capogna_LeDonne}.

\begin{theorem}[LD,  Capogna]\label{reg:theorem}
Let $F:M\to N$ be an  isometry between two    subRiemannian manifolds.
If there exist  two $C^\infty$ volume forms
$\vol_M$ and $\vol_N$   such that
  $F_*\vol_M=\vol_N$,
then 
$F$ is a $C^\infty$ diffeomorphism.
\end{theorem}

\begin{theorem}[LD,  Capogna]\label{Popp:preserved_intro}
Let $F:M\to N$ be an  isometry between equiregular subRiemannian manifolds.
If $\vol_M$ and $\vol_N$ are the Popp measures on $M$ and $N$, respectively,
then 
 $F_*\vol_M=\vol_N$.
\end{theorem}

\proof[How to prove Theorem~\ref{Popp:preserved_intro}]
In two steps:
\begin{itemize}
\item[Step 1.] Carnot group case: Popp is a Haar measure and hence a fixed multiple of the Hausdorff measure.
The latter is a metric invariant.
\item[Step 2.] There is a representation formula (\cite[pages 358-359]{Agrachev_Barilari_Boscain:Hausdorff},  \cite[Section 3.2]{ghezzi-jean}]) for the Popp volume in terms of the Hausdorff measure and the tangent measures in the tangent metric space, which is a Carnot group:
\begin{equation*}\label{F1}
\mathrm d \vol_M = 2^{-Q} {\mathcal N}_p(\vol_M) (B_{{\mathcal N}_p(M)} (e,1)) \mathrm d {\mathcal S}^Q_M.
\end{equation*}
One then makes use of Step 1.
 \qed
\end{itemize}

\proof[How to prove Theorem~\ref{reg:theorem}]
In many steps:
\begin{itemize}
\item[Step 1.]
$F$ is an isomorphism of metric {\em measure} spaces,
so $F$ preserves minimal upper gradients, Dirichlet energy, harmonic maps, and subLaplacian.
Recall the horizontal gradient:
If $X_1, \ldots, X_m$ is an orthonormal frame for the horizontal bundle, define $$\nabla_{\rm H} u := (X_1 u ) X_1 +\ldots +(X_m u ) X_m.$$
From the horizontal gradient one defines the subLaplacian:  $\Delta_{\rm H} u  = g$ means
 $$ \int_M g v  \;\mathrm  d\vol_M=\int_M \langle \nabla_{\rm H} u ,  \nabla_{\rm H} v\rangle \; \mathrm d\vol_M  , \qquad \forall v\in {\rm Lip}_c(M),
$$
We remark that  $\Delta_{\rm H}(\cdot)$ depends on the choice of the measure $\vol_M$.
\item[Step 2.] Hajlasz and Koskela's result,
\cite[page 51 and Section 11.2]{HK}:  
  $\norm{\nabla_{\rm H} u }$ coincides almost everywhere with the minimal upper gradient of $u$.
    Consequently, since  $F_*\vol_M=\vol_N$,
  $$ \Delta_{\rm H} u  = g  \implies   \Delta_{\rm H} (u\circ F)  = g\circ F,$$
  where the first subLaplacian is with respect to $\vol_N$ and the second one  with respect to $\vol_M$.
\item[Step 3.] Rothschild and Stein's version of H\"ormander's Hypoelliptic Theorem,  \cite[Theorem 18]{Roth:Stein}:
Let $X_0, X_1,...,X_r$ bracket generating  vector fields in $\R^n$.  
Let $u$ be a distributional solution to the equation $\mathcal   (X_0+ \sum_{i=1}^r X_i^2  )u=g$ and let   $k\in \N\cup \{0\}$ and $1<p<\infty$. Then, considering the horizontal local Sobolov spaces
$W^{k,p}_{\rm H,loc} $, 
$$g\in W_{\rm H}^{k,p}(\R^n,{\mathcal L}^n) \implies u\in W^{k+2,p}_{\rm H,loc}(\R^n,{\mathcal L}^n).$$
Corollary: 
If $\vol_M$ is a $C^\infty$ volume form,   
$$\Delta_{\rm H} u \in W^{k,p}_{\rm H,loc}(M,\vol_M)\implies u\in W^{k+1,p}_{\rm H,loc}(M,\vol_M).$$
\item[Step 4.]
Bootstrap  argument:
$F$ isometry implies that $F\in W^{1,p}_{\rm H}$. Taking $x_j$ some $C^\infty$ coordinate system, we get that
$\Delta_{\rm H} x_i=:g_i\in C^\infty \subset  W^{k,p}_{\rm H,loc}$, for all $k$ and $p$.
Then, by Step 3
$$\Delta_{\rm H} (x_i \circ F)  = g_i\circ F   \in W^{1,p}_{\rm H} 
\implies
%\stackrel{RS}{\implies} 
x_i \circ F   \in W^{2,p}_{\rm H}.$$
Then we iterate Rothschild and Stein's regularity of Step 3:
$$ g_i\circ F \in W^{2,p}_{\rm H,loc}  
\implies
%\stackrel{RS}{\implies} 
 x_i \circ F\in   W^{3,p}_{\rm H,loc}$$
and by induction we complete.
 \qed
\end{itemize}

\subsection{Sub-Riemannian isometries are determined by the horizontal differential}

\begin{corollary}\label{horizontal:differential}
Let $M$ and $N$ be two connected equiregular subRiemannian manifolds.
Let $p\in M$ and let $\Delta$ be the horizontal bundle of $M$.
Let $f,g:M\to N$ be two isometries.
If $f(p)=g(p)$ and $df|_{\Delta_p}=dg|_{\Delta_p}$, then $f=g$.
\end{corollary}

The proof can be  read in \cite[Proposition 2.8]{LeDonne-Ottazzi}.
Once we know that the isometries fixing a point are a (smooth) compact Lie group, the argument is an easy exercise in differential geometry.
\subsection{Isometries of nilpotent groups}

The fact that Carnot isometries are affine (Theorem~\ref{localisometries}) is a general feature of the fact that we are dealing with a nilpotent group.
In fact, isometries are affine whenever they are globally defined on a nilpotent connected group.
Here we obviously require that the distances are left-invariant and induce the manifold topology.
For example, this is the case for arbitrary homogeneous groups.

\begin{thm}[LD,  Kivioja]
Let \( N_1 \) and \( N_2 \) be two nilpotent connected metric Lie groups. Any isometry \( F \colon N_1 \ra N_2 \) is affine.
\end{thm}

This result is proved in  \cite{Kivioja_LeDonne_isom_nilpotent} with algebraic techniques by studying the nilradical, i.e., the biggest
%largest/biggest/maximal
nilpotent ideal, of the  group of self-isometries of a nilpotent connected metric Lie group. 
The proof leads back to a Riemannian result of Wolf, see \cite{}.

\subsection{Two isometric non-isomorphic groups}

In general   isometries of a subFinsler Lie group $G$ may not be affine, not even in the Riemannian setting.
As counterexample, we take the 
universal covering group $\tilde G$ of the
group $G=E(2)$ of Euclidean motions of the plane. 
This group is also called {\em roto-translation} group.
One can see that there exists a Riemannian distance on  $\tilde G$ %$\R^2\rtimes \mathbb S^1$ 
that makes it isometric to the Euclidean space $\R^3$. In particular, they have the same isometry group. However, a straightforward calculation of the automorphisms shows that not all isometries fixing the identity are group isomorphisms of $\tilde G$. %$\R^2\rtimes \mathbb S^1$. 
 As a side note, we remark that the group $\tilde G$ admit a left-invariant subRiemannian structure and a map into the subRiemannian Heisenberg group that is locally bi-Lipschitz. However, these two spaces are not quasi-conformal, see \cite{Faessler_Koskela_LeDonne}.
 
 Examples of isometric Lie groups that are not isomorphic can be found also in the strict subRiemannian context. There are examples in 3D, see  \cite{Agrachev_Barilari}. 
 Also the analogue of the roto-traslation construction can also be developed.
 %,  see \cite{Enrico&Seba, Cowling-LeDonne-Ottazzi}.

\noindent{\bf Acknowledgments.}
The author would like to thank
E.~Breuillard,
 S.~Nicolussi Golo,
 A.~Ottazzi,
 A.~Prantl, 
 S.~Rigot
 for help in preparing this article.
This primer   was written after the preparation of a  mini course held at the
 Ninth School on `Analysis and Geometry in Metric Spaces'
in Levico Terme  in July 2015.
 The author would like to also thank
  the organizers: L.~Ambrosio, 
B.~Franchi, 
I.~Markina, 
R.~Serapioni, 
F.~Serra Cassano. 
  The author  is supported by the Academy of Finland project no. 288501.   
 
\bibliography{general_bibliography} 

\def\cprime{$'$} \def\cprime{$'$} \def\cprime{$'$} \def\cprime{$'$}
  \def\cprime{$'$}
\providecommand{\bysame}{\leavevmode\hbox to3em{\hrulefill}\thinspace}
\providecommand{\MR}{\relax\ifhmode\unskip\space\fi MR }
\providecommand{\MRhref}[2]{%
  \href{http://www.ams.org/mathscinet-getitem?mr=#1}{#2}
}
\providecommand{\href}[2]{#2}
\begin{thebibliography}{MZGZ09}

\bibitem[AB12]{Agrachev_Barilari}
Andrei Agrachev and Davide Barilari, \emph{Sub-{R}iemannian structures on 3{D}
  {L}ie groups}, J. Dyn. Control Syst. \textbf{18} (2012), no.~1, 21--44.

\bibitem[ABB12]{Agrachev_Barilari_Boscain:Hausdorff}
Andrei Agrachev, Davide Barilari, and Ugo Boscain, \emph{On the {H}ausdorff
  volume in sub-{R}iemannian geometry}, Calc. Var. Partial Differential
  Equations \textbf{43} (2012), no.~3-4, 355--388.

\bibitem[ABB15]{Agrachev_Barilari_Boscain:book}
\bysame, \emph{Introduction to {R}iemannian and {S}ub-{R}iemannian geometry},
  Manuscript (2015).

\bibitem[AK00a]{Ambrosio-Kirchheim}
Luigi Ambrosio and Bernd Kirchheim, \emph{Currents in metric spaces}, Acta
  Math. \textbf{185} (2000), no.~1, 1--80.

\bibitem[AK00b]{Ambrosio-Kirchheim:rect}
\bysame, \emph{Rectifiable sets in metric and {B}anach spaces}, Math. Ann.
  \textbf{318} (2000), no.~3, 527--555.

\bibitem[AKL09]{AKL}
Luigi Ambrosio, Bruce Kleiner, and Enrico {L}{e Donne}, \emph{Rectifiability of
  sets of finite perimeter in {C}arnot groups: existence of a tangent
  hyperplane}, J. Geom. Anal. \textbf{19} (2009), no.~3, 509--540.

\bibitem[Amb01]{Amb02}
Luigi Ambrosio, \emph{Some fine properties of sets of finite perimeter in
  {A}hlfors regular metric measure spaces}, Adv. Math. \textbf{159} (2001),
  no.~1, 51--67.

\bibitem[Amb02]{Amb01}
\bysame, \emph{Fine properties of sets of finite perimeter in doubling metric
  measure spaces}, Set-Valued Anal. \textbf{10} (2002), no.~2-3, 111--128,
  Calculus of variations, nonsmooth analysis and related topics.

\bibitem[ASV06]{ASCV}
Luigi Ambrosio, Francesco {S}{erra Cassano}, and Davide Vittone,
  \emph{Intrinsic regular hypersurfaces in {H}eisenberg groups}, J. Geom. Anal.
  \textbf{16} (2006), no.~2, 187--232.

\bibitem[BBI01]{Burago:book}
Dmitri{\u\i} Burago, Yuri{\u\i} Burago, and Sergei{\u\i} Ivanov, \emph{A course
  in metric geometry}, Graduate Studies in Mathematics, vol.~33, American
  Mathematical Society, Providence, RI, 2001.

\bibitem[Bel96]{bellaiche}
Andr{\'e} Bella{\"{\i}}che, \emph{The tangent space in sub-{R}iemannian
  geometry}, Sub-Riemannian geometry, Progr. Math., vol. 144, Birkh\"auser,
  Basel, 1996, pp.~1--78.

\bibitem[{\noopsort{Be}}L13]{Bellettini-LeDonne}
Costante {\noopsort{Be}}llettini and Enrico {L}{e Donne}, \emph{Regularity of
  sets with constant horizontal normal in the {E}ngel group}, Comm. Anal. Geom.
  \textbf{21} (2013), no.~3, 469--507.

\bibitem[Ber88]{b}
Valeri{\u\i}~N. Berestovski{\u\i}, \emph{Homogeneous manifolds with an
  intrinsic metric. {I}}, Sibirsk. Mat. Zh. \textbf{29} (1988), no.~6, 17--29.

\bibitem[BGP92]{Burago-Gromov-Perelman}
Yuri{\u\i} Burago, Mikhail Gromov, and Grigori{\u\i} Perel{\cprime}man,
  \emph{A. {D}. {A}leksandrov spaces with curvatures bounded below}, Uspekhi
  Mat. Nauk \textbf{47} (1992), no.~2(284), 3--51, 222.

\bibitem[BK02]{Bonk-Kleiner}
Mario Bonk and Bruce Kleiner, \emph{Quasisymmetric parametrizations of
  two-dimensional metric spheres}, Invent. Math. \textbf{150} (2002), no.~1,
  127--183.

\bibitem[BLU07]{Bonfiglioli:et:al}
A.~Bonfiglioli, E.~Lanconelli, and F.~Uguzzoni, \emph{Stratified {L}ie groups
  and potential theory for their sub-{L}aplacians}, Springer Monographs in
  Mathematics, Springer, Berlin, 2007.

\bibitem[BM91]{Bestvina-Mess}
Mladen Bestvina and Geoffrey Mess, \emph{The boundary of negatively curved
  groups}, J. Amer. Math. Soc. \textbf{4} (1991), no.~3, 469--481.

\bibitem[BP00]{Bourdon_Pajot_2000}
Marc Bourdon and Herv{\'e} Pajot, \emph{Rigidity of quasi-isometries for some
  hyperbolic buildings}, Comment. Math. Helv. \textbf{75} (2000), no.~4,
  701--736.

\bibitem[BR13]{Barilari-Rizzi}
Davide Barilari and Luca Rizzi, \emph{A formula for {P}opp's volume in
  sub-{R}iemannian geometry}, Anal. Geom. Metr. Spaces \textbf{1} (2013),
  42--57.

\bibitem[{\noopsort{Br}}L]{Breuillard-LeDonne1_1}
Emmanuel {\noopsort{Br}}euillard and Enrico {L}{e Donne}, \emph{Nilpotent
  groups, asymptotic cones and sub{F}insler geometry}, In preparation.

\bibitem[{\noopsort{Br}}L13]{Breuillard-LeDonne1}
\bysame, \emph{On the rate of convergence to the asymptotic cone for nilpotent
  groups and sub{F}insler geometry}, Proc. Natl. Acad. Sci. USA \textbf{110}
  (2013), no.~48, 19220--19226.

\bibitem[BSV07]{Barone-Adesi-Serra-Cassano-Vittone}
Vittorio {B}{arone Adesi}, Francesco {S}{erra Cassano}, and Davide Vittone,
  \emph{The {B}ernstein problem for intrinsic graphs in {H}eisenberg groups and
  calibrations}, Calc. Var. Partial Differential Equations \textbf{30} (2007),
  no.~1, 17--49.

\bibitem[Car09]{Caratheodory}
C.~Carath{\'e}odory, \emph{Untersuchungen \"uber die {G}rundlagen der
  {T}hermodynamik}, Math. Ann. \textbf{67} (1909), no.~3, 355--386.

\bibitem[CC97]{Cheeger-Colding}
Jeff Cheeger and Tobias~H. Colding, \emph{On the structure of spaces with
  {R}icci curvature bounded below. {I}}, J. Differential Geom. \textbf{46}
  (1997), no.~3, 406--480.

\bibitem[CC06]{Capogna-Cowling}
Luca Capogna and Michael Cowling, \emph{Conformality and {$Q$}-harmonicity in
  {C}arnot groups}, Duke Math. J. \textbf{135} (2006), no.~3, 455--479.

\bibitem[CDPT07]{Capogna-et-al}
Luca Capogna, Donatella Danielli, Scott~D. Pauls, and Jeremy~T. Tyson, \emph{An
  introduction to the {H}eisenberg group and the sub-{R}iemannian isoperimetric
  problem}, Progress in Mathematics, vol. 259, Birkh\"auser Verlag, Basel,
  2007.

\bibitem[CG90]{Corwin-Greenleaf}
Lawrence~J. Corwin and Frederick~P. Greenleaf, \emph{Representations of
  nilpotent {L}ie groups and their applications. {P}art {I}}, Cambridge Studies
  in Advanced Mathematics, vol.~18, Cambridge University Press, Cambridge,
  1990, Basic theory and examples.

\bibitem[Che99]{Cheeger}
Jeff Cheeger, \emph{Differentiability of {L}ipschitz functions on metric
  measure spaces}, Geom. Funct. Anal. \textbf{9} (1999), no.~3, 428--517.

\bibitem[CK06]{Cheeger-Kleiner}
Jeff Cheeger and Bruce Kleiner, \emph{On the differentiability of {L}ipschitz
  maps from metric measure spaces to {B}anach spaces}, Inspired by S. S. Chern,
  Nankai Tracts Math., vol.~11, World Sci. Publ., Hackensack, NJ, 2006,
  pp.~129--152.

\bibitem[CK10a]{Cheeger-Kleiner2}
\bysame, \emph{Differentiating maps into {$L^1$}, and the geometry of {BV}
  functions}, Ann. of Math. (2) \textbf{171} (2010), no.~2, 1347--1385.

\bibitem[CK10b]{Cheeger-Kleiner3}
\bysame, \emph{Metric differentiation, monotonicity and maps to {$L^1$}},
  Invent. Math. \textbf{182} (2010), no.~2, 335--370.

\bibitem[CL14]{Capogna_LeDonne}
Luca Capogna and Enrico {L}{e Donne}, \emph{Smoothness of sub{R}iemannian
  isometries}, Accepted in the American Journal of Mathematics (2014).

\bibitem[CP06]{colepauls}
Daniel~R. Cole and Scott~D. Pauls, \emph{{$C\sp 1$} hypersurfaces of the
  {H}eisenberg group are {$N$}-rectifiable}, Houston J. Math. \textbf{32}
  (2006), no.~3, 713--724 (electronic). \MR{MR2247905 (2007f:53032)}

\bibitem[DGN08]{Danielli-Garofalo-Nhieu}
D.~Danielli, N.~Garofalo, and D.~M. Nhieu, \emph{A notable family of entire
  intrinsic minimal graphs in the {H}eisenberg group which are not perimeter
  minimizing}, Amer. J. Math. \textbf{130} (2008), no.~2, 317--339.

\bibitem[FKL14]{Faessler_Koskela_LeDonne}
Katrin Fa\"ssler, Pekka Koskela, and Enrico {L}{e Donne}, \emph{Nonexistence of
  quasiconformal maps between certain metric measure spaces}, Int. Math. Res.
  Not. IMRN, to appear (2014).

\bibitem[FS82]{Folland-Stein}
G.~B. Folland and Elias~M. Stein, \emph{Hardy spaces on homogeneous groups},
  Mathematical Notes, vol.~28, Princeton University Press, Princeton, N.J.,
  1982.

\bibitem[FSS03]{fssc}
Bruno Franchi, Raul Serapioni, and Francesco {S}{erra Cassano}, \emph{On the
  structure of finite perimeter sets in step 2 {C}arnot groups}, J. Geom. Anal.
  \textbf{13} (2003), no.~3, 421--466.

\bibitem[GJ14]{ghezzi-jean}
Roberta Ghezzi and Fr{\'e}d{\'e}ric Jean, \emph{Hausdorff measure and
  dimensions in non equiregular sub-{R}iemannian manifolds}, Geometric control
  theory and sub-{R}iemannian geometry \textbf{5} (2014), 201--218.

\bibitem[Gro81]{Gromov-polygrowth}
Mikhael Gromov, \emph{Groups of polynomial growth and expanding maps}, Inst.
  Hautes \'Etudes Sci. Publ. Math. (1981), no.~53, 53--73.

\bibitem[Gro96]{Gromov1}
Mikhail Gromov, \emph{{C}arnot-{C}arath\'eodory spaces seen from within},
  Sub-Riemannian geometry, Progr. Math., vol. 144, Birkh\"auser, Basel, 1996,
  pp.~79--323.

\bibitem[Gro99]{Gromov}
\bysame, \emph{Metric structures for {R}iemannian and non-{R}iemannian spaces},
  Progress in Mathematics, vol. 152, Birkh\"auser Boston Inc., Boston, MA,
  1999, Based on the 1981 French original, With appendices by M.\ Katz, P.\
  Pansu and S.\ Semmes, Translated from the French by Sean Michael Bates.

\bibitem[GS92]{Gromov-Schoen}
Mikhail Gromov and Richard Schoen, \emph{Harmonic maps into singular spaces and
  {$p$}-adic superrigidity for lattices in groups of rank one}, Inst. Hautes
  \'Etudes Sci. Publ. Math. (1992), no.~76, 165--246.

\bibitem[Ham90]{Hamenstadt}
Ursula Hamenst{\"a}dt, \emph{Some regularity theorems for
  {C}arnot-{C}arath\'eodory metrics}, J. Differential Geom. \textbf{32} (1990),
  no.~3, 819--850.

\bibitem[Hei95]{Heinonen-calculus}
Juha Heinonen, \emph{Calculus on {C}arnot groups}, Fall {S}chool in {A}nalysis
  ({J}yv\"askyl\"a, 1994), Report, vol.~68, Univ. Jyv\"askyl\"a, Jyv\"askyl\"a,
  1995, pp.~1--31.

\bibitem[Hei01]{Heinonenbook}
\bysame, \emph{Lectures on analysis on metric spaces}, Universitext,
  Springer-Verlag, New York, 2001.

\bibitem[HK98]{Heinonen-Koskela}
Juha Heinonen and Pekka Koskela, \emph{Quasiconformal maps in metric spaces
  with controlled geometry}, Acta Math. \textbf{181} (1998), no.~1, 1--61.

\bibitem[HK00]{HK}
Piotr Haj{\l}asz and Pekka Koskela, \emph{Sobolev met {P}oincar\'e}, Mem. Amer.
  Math. Soc. \textbf{145} (2000), no.~688, x+101.

\bibitem[HS90]{Hebisch-Sikora}
Waldemar Hebisch and Adam Sikora, \emph{A smooth subadditive homogeneous norm
  on a homogeneous group}, Studia Math. \textbf{96} (1990), no.~3, 231--236.

\bibitem[Jea14]{jeancontrol}
Fr{\'e}d{\'e}ric Jean, \emph{Control of nonholonomic systems: from
  sub-{R}iemannian geometry to motion planning}, Springer Briefs in
  Mathematics, Springer, Cham, 2014.

\bibitem[Jer86]{Jerison}
David Jerison, \emph{The {P}oincar\'e inequality for vector fields satisfying
  {H}\"ormander's condition}, Duke Math. J. \textbf{53} (1986), no.~2,
  503--523.

\bibitem[KB02]{Kapovich-Benakli}
Ilya Kapovich and Nadia Benakli, \emph{Boundaries of hyperbolic groups},
  Combinatorial and geometric group theory (New York, 2000/Hoboken, NJ, 2001),
  Contemp. Math., vol. 296, Amer. Math. Soc., Providence, RI, 2002, pp.~39--93.

\bibitem[KL97]{Kleiner-Leeb}
Bruce Kleiner and Bernhard Leeb, \emph{Rigidity of quasi-isometries for
  symmetric spaces and {E}uclidean buildings}, Inst. Hautes \'Etudes Sci. Publ.
  Math. (1997), no.~86, 115--197 (1998).

\bibitem[KL16]{Kivioja_LeDonne_isom_nilpotent}
Ville Kivioja and Enrico {L}{e Donne}, \emph{Isometries of nilpotent metric
  groups}, Preprint, submitted, arXiv:1601.08172 (2016).

\bibitem[KR85]{Koranyi-Reimann85}
A.~Kor{\'a}nyi and H.~M. Reimann, \emph{Quasiconformal mappings on the
  {H}eisenberg group}, Invent. Math. \textbf{80} (1985), no.~2, 309--338.

\bibitem[KR95]{Koranyi-Reimann95}
\bysame, \emph{Foundations for the theory of quasiconformal mappings on the
  {H}eisenberg group}, Adv. Math. \textbf{111} (1995), no.~1, 1--87.

\bibitem[KS04]{Kirchheim-SerraCassano}
Bernd Kirchheim and Francesco {S}{erra Cassano}, \emph{Rectifiability and
  parameterization of intrinsic regular surfaces in the {H}eisenberg group},
  Ann. Sc. Norm. Super. Pisa Cl. Sci. (5) \textbf{3} (2004), no.~4, 871--896.

\bibitem[Laa02]{Laakso}
Tomi~J. Laakso, \emph{Plane with {$A\sb \infty$}-weighted metric not
  bi-{L}ipschitz embeddable to {${\mathbb R}\sp N$}}, Bull. London Math. Soc.
  \textbf{34} (2002), no.~6, 667--676.

\bibitem[LD13]{LeDonne_PIE}
Enrico Le~Donne, \emph{Lipschitz and path isometric embeddings of metric
  spaces}, Geom. Dedicata \textbf{166} (2013), 47--66.

\bibitem[LD15a]{LeDonne-lectures}
\bysame, \emph{Lecture notes on sub-{R}iemannian geometry}, Manuscript (2015).

\bibitem[LD15b]{LeDonne_characterization}
\bysame, \emph{A metric characterization of {C}arnot groups}, Proc. Amer. Math.
  Soc. \textbf{143} (2015), no.~2, 845--849.

\bibitem[LN15]{LeDonne_NicolussiGolo}
Enrico {L}{e Donne} and Sebastiano {N}{icolussi Golo}, \emph{Regularity
  properties of spheres in homogeneous groups}, Preprint, submitted,
  arXiv:1509.03881 (2015).

\bibitem[LO16]{LeDonne-Ottazzi}
Enrico {L}{e Donne} and Alessandro Ottazzi, \emph{Isometries of {C}arnot
  {G}roups and {S}ub-{F}insler {H}omogeneous {M}anifolds}, J. Geom. Anal.
  \textbf{26} (2016), no.~1, 330--345. \MR{3441517}

\bibitem[LP01]{Lang-Plaut}
Urs Lang and Conrad Plaut, \emph{Bilipschitz embeddings of metric spaces into
  space forms}, Geom. Dedicata \textbf{87} (2001), no.~1-3, 285--307.

\bibitem[LR14]{LeDonne_Rigot_Heisenberg_BCP}
Enrico {L}{e Donne} and S{\'e}verine Rigot, \emph{Besicovitch {C}overing
  {P}roperty for homogeneous distances on the {H}eisenberg groups}, Preprint,
  submitted, arXiv:1406.1484 (2014).

\bibitem[LR15]{LeDonne_Rigot_BCP_graded_groups}
\bysame, \emph{Besicovitch {C}overing {P}roperty on graded groups and
  applications to measure differentiation}, Preprint, submitted,
  arXiv:1512.04936 (2015).

\bibitem[LZ13]{LeDonne-Zust}
Enrico {L}{e Donne} and Roger Z{\"u}st, \emph{Some properties of {H}\"older
  surfaces in the {H}eisenberg group}, Illinois J. Math. \textbf{57} (2013),
  no.~1, 229--249.

\bibitem[Mag02]{magntesi}
Valentino Magnani, \emph{Elements of geometric measure theory on
  sub-{R}iemannian groups}, Scuola Normale Superiore, Pisa, 2002.

\bibitem[Mag08]{Magnani08}
\bysame, \emph{Non-horizontal submanifolds and coarea formula}, J. Anal. Math.
  \textbf{106} (2008), 95--127.

\bibitem[Mit85]{Mitchell}
John Mitchell, \emph{On {C}arnot-{C}arath\'eodory metrics}, J. Differential
  Geom. \textbf{21} (1985), no.~1, 35--45.

\bibitem[MM95]{Margulis-Mostow}
Gregori~A. Margulis and George~D. Mostow, \emph{The differential of a
  quasi-conformal mapping of a {C}arnot-{C}arath\'eodory space}, Geom. Funct.
  Anal. \textbf{5} (1995), no.~2, 402--433.

\bibitem[MM00]{Margulis-Mostow2}
G.~A. Margulis and G.~D. Mostow, \emph{Some remarks on the definition of
  tangent cones in a {C}arnot-{C}arath\'eodory space}, J. Anal. Math.
  \textbf{80} (2000), 299--317.

\bibitem[Mon02]{Montgomery}
Richard Montgomery, \emph{A tour of subriemannian geometries, their geodesics
  and applications}, Mathematical Surveys and Monographs, vol.~91, American
  Mathematical Society, Providence, RI, 2002.

\bibitem[MSV08]{MSCV}
Roberto Monti, Francesco {S}{erra Cassano}, and Davide Vittone, \emph{A
  negative answer to the {B}ernstein problem for intrinsic graphs in the
  {H}eisenberg group}, Boll. Unione Mat. Ital. (9) \textbf{1} (2008), no.~3,
  709--727.

\bibitem[MV08]{Magnani-Vittone}
Valentino Magnani and Davide Vittone, \emph{An intrinsic measure for
  submanifolds in stratified groups}, J. Reine Angew. Math. \textbf{619}
  (2008), 203--232.

\bibitem[MZ74]{mz}
Deane Montgomery and Leo Zippin, \emph{Topological transformation groups},
  Robert E. Krieger Publishing Co., Huntington, N.Y., 1974, Reprint of the 1955
  original.

\bibitem[MZGZ09]{MR2531368}
Giuseppe Mingione, Anna Zatorska-Goldstein, and Xiao Zhong, \emph{Gradient
  regularity for elliptic equations in the {H}eisenberg group}, Adv. Math.
  \textbf{222} (2009), no.~1, 62--129.

\bibitem[NSW85]{nagelstwe}
Alexander Nagel, Elias~M. Stein, and Stephen Wainger, \emph{Balls and metrics
  defined by vector fields. {I}. {B}asic properties}, Acta Math. \textbf{155}
  (1985), no.~1-2, 103--147.

\bibitem[Pan82a]{PansuThesis}
Pierre Pansu, \emph{G\'eometie du goupe d'{H}eisenberg}, Thesis, Universite
  Paris VII (1982).

\bibitem[Pan82b]{pansucras}
\bysame, \emph{Une in\'egalit\'e isop\'erim\'etrique sur le groupe de
  {H}eisenberg}, C. R. Acad. Sci. Paris S\'er. I Math. \textbf{295} (1982),
  no.~2, 127--130.

\bibitem[Pan83a]{Pansu83}
P.~Pansu, \emph{An isoperimetric inequality on the {H}eisenberg group}, Rend.
  Sem. Mat. Univ. Politec. Torino (1983), no.~Special Issue, 159--174 (1984),
  Conference on differential geometry on homogeneous spaces (Turin, 1983).

\bibitem[Pan83b]{Pansu-croissance}
Pierre Pansu, \emph{Croissance des boules et des g\'eod\'esiques ferm\'ees dans
  les nilvari\'et\'es}, Ergodic Theory Dynam. Systems \textbf{3} (1983), no.~3,
  415--445.

\bibitem[Pan89]{Pansu}
\bysame, \emph{M\'etriques de {C}arnot-{C}arath\'eodory et quasiisom\'etries
  des espaces sym\'etriques de rang un}, Ann. of Math. (2) \textbf{129} (1989),
  no.~1, 1--60.

\bibitem[Pau04]{Pauls-minimal}
Scott~D. Pauls, \emph{Minimal surfaces in the {H}eisenberg group}, Geom.
  Dedicata \textbf{104} (2004), 201--231.

\bibitem[Rif14]{Rifford:book}
Ludovic Rifford, \emph{Sub-{R}iemannian geometry and optimal transport},
  Springer Briefs in Mathematics, Springer, Cham, 2014.

\bibitem[Rit09]{Ritore}
Manuel Ritor{\'e}, \emph{Examples of area-minimizing surfaces in the
  sub-{R}iemannian {H}eisenberg group {$\mathbb{H}^1$} with low regularity},
  Calc. Var. Partial Differential Equations \textbf{34} (2009), no.~2,
  179--192.

\bibitem[RR08]{Ritore-Rosales08}
Manuel Ritor{\'e} and C{\'e}sar Rosales, \emph{Area-stationary surfaces in the
  {H}eisenberg group {$\mathbb{H}^1$}}, Adv. Math. \textbf{219} (2008), no.~2,
  633--671.

\bibitem[RS76]{Roth:Stein}
Linda~Preiss Rothschild and E.~M. Stein, \emph{Hypoelliptic differential
  operators and nilpotent groups}, Acta Math. \textbf{137} (1976), no.~3-4,
  247--320.

\bibitem[SC16]{SerraCassanoGMT}
Francesco Serra~Cassano, \emph{Some topics of geometric measure theory in
  carnot groups}, Dynamics, geometry and analysis on sub-Riemannian manifolds
  (2016), To appear.

\bibitem[Sem96a]{Semmes}
Stephen Semmes, \emph{Good metric spaces without good parameterizations}, Rev.
  Mat. Iberoamericana \textbf{12} (1996), no.~1, 187--275.

\bibitem[Sem96b]{Semmes2}
\bysame, \emph{On the nonexistence of bi-{L}ipschitz parameterizations and
  geometric problems about {$A\sb \infty$}-weights}, Rev. Mat. Iberoamericana
  \textbf{12} (1996), no.~2, 337--410.

\bibitem[Sie86]{Siebert}
Eberhard Siebert, \emph{Contractive automorphisms on locally compact groups},
  Math. Z. \textbf{191} (1986), no.~1, 73--90.

\bibitem[Ste93]{Stein:book}
Elias~M. Stein, \emph{Harmonic analysis: real-variable methods, orthogonality,
  and oscillatory integrals}, Princeton Mathematical Series, vol.~43, Princeton
  University Press, Princeton, NJ, 1993, With the assistance of Timothy S.
  Murphy, Monographs in Harmonic Analysis, III. \MR{MR1232192 (95c:42002)}

\bibitem[Str86]{Strichartz}
Robert~S. Strichartz, \emph{Sub-{R}iemannian geometry}, J. Differential Geom.
  \textbf{24} (1986), no.~2, 221--263.

\bibitem[Vit08]{Vittone-tesi}
Davide Vittone, \emph{Submanifolds in {C}arnot groups}, Tesi. Scuola Normale
  Superiore di Pisa (Nuova Series) [Theses of Scuola Normale Superiore di Pisa
  (New Series)], vol.~7, Edizioni della Normale, Pisa, 2008, Thesis, Scuola
  Normale Superiore, Pisa, 2008.

\bibitem[VSCC92]{varsalcou}
N.~Th. Varopoulos, L.~Saloff-Coste, and T.~Coulhon, \emph{Analysis and geometry
  on groups}, Cambridge Tracts in Mathematics, vol. 100, Cambridge University
  Press, Cambridge, 1992. \MR{MR1218884 (95f:43008)}

\bibitem[War83]{Warner}
Frank~W. Warner, \emph{Foundations of differentiable manifolds and {L}ie
  groups}, Graduate Texts in Mathematics, vol.~94, Springer-Verlag, New York,
  1983, Corrected reprint of the 1971 edition.

\end{thebibliography}
\bibliographystyle{amsalpha}

\end{document}